\documentclass[a4paper]{article}
\usepackage{amsmath,amsfonts,amssymb,amsthm,amstext}
\usepackage{mathtools}
\usepackage{latexsym}
\usepackage{microtype}
\usepackage{lmodern}
\usepackage[T1]{fontenc}
\usepackage[utf8]{inputenc}
\usepackage[english]{babel}
\xdef\articletitle{On Stability and Convergence of a Three-layer Semi-discrete Scheme for an Abstract Analogue of the Ball Integro-differential Equation}
\xdef\Rogava{Jemal Rogava}
\xdef\Tsiklauri{Mikheil Tsiklauri}
\xdef\Vashakidze{Zurab Vashakidze}
\usepackage[pdftex,
 pdfauthor={ {\Rogava}, {\Tsiklauri} and {\Vashakidze} },
 pdftitle={\articletitle},
 pdfsubject={Mathematics, Applied Mathematics},
 pdfstartview={FitH}]{hyperref}
\usepackage{tikz}
\usepackage{xcolor}
\usepackage{enumitem}
\usepackage{graphicx}
\usepackage{float}
\usepackage{color}
\usepackage{subcaption}
\usepackage{pgf,pgfplots,pgfplotstable}
\usepackage{array}
\usepackage{booktabs}
\usepackage[nottoc]{tocbibind}
\usepackage[square,sort,comma,numbers,sectionbib]{natbib}
\graphicspath{ {figures/} }
\usepackage{geometry}
\pgfplotsset{compat=newest}

\numberwithin{equation}{section}

\theoremstyle{plain}
\newtheorem{theorem}{Theorem}[section]

\newtheorem{lemma}{Lemma}[section]
\newtheorem{remark}{Remark}[section]

\newtheorem*{definition*}{Definition}

\makeatletter
\renewenvironment{proof}[1][\proofname]{%
	\par\pushQED{\qed}\normalfont%
	\topsep6\p@\@plus6\p@\relax
	\trivlist\item[\hskip\labelsep\bfseries#1\@addpunct{.}]%
	\ignorespaces
}{%
	\popQED\endtrivlist\@endpefalse
}
\makeatother

\providecommand{\keywords}[1]
{
	\small	
	\noindent\textbf{\textit{Keywords and phrases:}} #1
}

\providecommand{\msc}[1]
{
	\small	
	\noindent\textbf{\textit{MSC 2010:}} #1
}

\definecolor{lime}{HTML}{A6CE39}
\DeclareRobustCommand{\orcidicon}{
\begin{tikzpicture}
	\draw[lime, fill=lime] (0,0) 
	circle [radius=0.16] 
	node[white] {{\fontfamily{qag}\selectfont \tiny ID}};
	\draw[white, fill=white] (-0.0625,0.095) 
	circle [radius=0.007];
\end{tikzpicture}
\hspace{-2mm}
}
\foreach \x in {A, ..., Z}{%
	\expandafter\xdef\csname orcid\x\endcsname{\noexpand\href{https://orcid.org/\csname orcidauthor\x\endcsname}{\noexpand\orcidicon}}
}

\title{\articletitle}
\author{ {\Rogava}\orcidA{}, {\Tsiklauri}\orcidB{} and {\Vashakidze}\orcidC{} }

\date{}

\begin{document}

\maketitle

\begin{abstract}\label{abstract}
	We consider the Cauchy problem for a second-order nonlinear evolution equation in a Hilbert space. This equation represents the abstract generalization of the Ball integro-differential equation. The general nonlinear case with respect to terms of the equation which include a square of a norm of a gradient is considered. A three-layer semi-discrete scheme is proposed in order to find an approximate solution. In this scheme, the approximation of nonlinear terms that are dependent on the gradient is carried out by using an integral mean. We show that the solution of the nonlinear discrete problem and its corresponding difference analogue of a first-order derivative is uniformly bounded. For the solution of the corresponding linear discrete problem, it is obtained high-order \textit{a priori} estimates by using two-variable Chebyshev polynomials. Based on these estimates we prove the stability of the nonlinear discrete problem. For smooth solutions, we provide error estimates for the approximate solution. An iteration method is applied in order to find an approximate solution for each temporal step. The convergence of the iteration process is proved.
\end{abstract}
 \\
\msc{46N40, 65J08, 65M06, 65M12, 65M22, 74H15, 74K10.}

\section*{Introduction}
In the present work, we consider the Cauchy problem in the Hilbert space for a nonlinear second-order abstract differential equation. Coefficients in the main part of the equation are self-adjoint positively defined, in general, unbounded operators. Our goal is to find an approximate solution to this problem. To do so, we apply a three-layer symmetrical semi-discrete scheme. In this scheme, nonlinear terms are approximated by using integral mean.

The considered equation represents an abstract generalization of J. M. Ball beam equation (see \cite{BL}). J. M. Ball has generalized Kirchhoff type nonlinear equation for beam, that was obtained by S. Woinowsky-Krieger (see \cite{S0}), by introducing damping terms, in order to account for the effect of external and internal damping.

Investigation of the topics related to the classic Kirchhoff equation started with Bernstein's well-known paper (see \cite{Ber}). In this paper existence and uniqueness issues for local as well as global solutions of initial-boundary value problem for the Kirchhoff string equation is studied. The issues of solvability of the classical and generalized Kirchhoff equations were later considered by many authors: A. Arosio, S. Panizzi \cite{Ar}, L. Berselli, R. Manfrin \cite{BM}, P. D'Ancona, S. Spagnolo \cite{DanSp2}, \cite{DanSp}, R. Manfrin \cite{Man}, L. A. Medeiros \cite{Med}, M. Matos \cite{Mat}, K. Nishihara \cite{Nish}, S. Panizzi \cite{Pan}. In the works \cite{Ar}, \cite{DanSp2}, \cite{DanSp}, \cite{Man} and \cite{Nish} issues of well-posedness and global solvability are thoroughly studied for a generalized Kirchhoff equation. In \cite{Pan} the existence of a global solution with low regularity is studied for Kirchhoff-type equations. An abstract analogue of the Kirchhoff-type beam equation is considered in the work by L. A. Medeiros \cite{Med}, where the existence and uniqueness theorem for the regular solution of the Cauchy problem is proved. The same abstract nonlinear equation, strengthened by the first derivative with respect to time, is discussed
in the work by P. Biler and E. H. de Brito (see \cite{Bi}, \cite{Br}), where most attention is paid to study of the behaviour of Cauchy problem. We should note that participation of the square of the main operator in the linear
part of this equation essentially helps to obtain the necessary \textit{a priori} estimates.

The following works are dedicated to approximate solutions of initial-boundary value problems for classical and generalized Kirchhoff equations: A. I. Christie, J. Sanz-Serna \cite{SS}, T. Geveci, I. Christie \cite{GCh}, I.-S. Liu, M.A. Rincon \cite{LR}, J. Peradze \cite{Per}, J. Rogava, M. Tsiklauri \cite{RTs}, \cite{RTs2} and in \cite{vashakidze2020application}. An algorithm of approximate solution for the dynamic beam equation is studied in \cite{GCh}. This algorithm represents a combination of the Galerkin method for spatial coordinates and the finite difference method for the time coordinate. The same combination of the methods is investigated for the classic Kirchhoff equation in  \cite{Per}. Design of algorithms for finding numerical solutions and their investigations for initial-boundary value problems of some classes integro-differential equations are considered in the book of T. Jangveladze, Z. Kiguradze and B. Neta \cite{JANGVELADZE201669}.

As it was mentioned before J. M. Ball - generalized the Kirchhoff beam equation by introducing damping terms, to account for the effect of external and internal damping. For an approximate solution of the initial-boundary problem of this equation, S. M. Choo and S. K. Chung proposed the finite difference method (see \cite{CC}). In this work stability and convergence of the approximate solution is investigated.

As far as we know, issues of approximate solution of abstract analogue of Kirchhoff-type equation for a beam are less studied. In the present paper, investigations of stability and convergence of the designed semi-discrete scheme for second-order (complete kind) nonlinear operator differential equation that represents the abstract analogue of a model of J. M. Ball for the beam is based on two facts: (a) $\left (u_{k} -u_{k -1}\right )/\tau $ and $B^{1/2} u_{k}$ are uniformly bounded ($u_{k}$ is an approximate solution, and $\tau $ is time step; linear operator $B$ is included in the main part of the equation); (b) For the solution of the corresponding linear problem an \textit{a priori} estimation is obtained where on the left-hand side power $s$ and on the right-hand side power $s -1$ of the operator $B$ is included. These facts give the possibility to weaken the nonlinear terms in the given nonlinear equation so much that, to make it possible to apply Gr\"{o}nwall's lemma. Besides, it is not required to impose any essential restriction for the temporal step $\tau$.

\section{Statement of the problem and semi-discrete scheme}\label{sec:section1}
Let us consider the following Cauchy problem in Hilbert space $H$:
\begin{align}
	&\frac{d^{2} u}{d t^{2}} +a_{1} B \frac{d u}{d t} +a_{2} B u +\psi _{1} \left (\left \Vert A^{1/2} u\right \Vert ^{2}\right ) A u \nonumber \\
	&+\frac{d}{d t} \left (\psi _{2} \left (\left \Vert A^{1/2} u\right \Vert ^{2}\right )\right ) A u +\psi _{3} \left (\left \Vert u\right \Vert ^{2}\right ) u \nonumber \\
	&+C u +N \frac{d u}{d t} +M \left (u\right ) = f (t) ,\text{\quad }t \in ]0 ,\;\overline{t}] , \label{E1.1} \\
	&u \left (0\right ) = \varphi _{0} ,\text{\quad }u^{ \prime } \left (t\right )\vert _{t =0} =\varphi _{1}\;  \label{E1.2}
\end{align}
where $A$ and $B$ are self-adjoint, positively defined (generally unbounded) operators with the domains $D \left (A\right )$ and $D \left (B\right )$ which are everywhere dense in $H$, besides, the following conditions are fulfilled
\begin{equation}\left \Vert A u\right \Vert ^{2} \leq b_{0}^{2} \left (B u ,u\right ) ,\text{\quad } \forall u \in D (B) \subset D \left (A\right ) ,\text{\quad }b_{0} =c o n s t >0 , \label{E1.2.1}
\end{equation}
where by $\Vert  \cdot \Vert $ and $\left ( \cdot  , \cdot \right )$ are defined correspondingly the norm and scalar product in $H$; $\psi _{1} (s)$, $\;\psi _{2} (s)$ and $\psi _{3} (s)$,$\;s \in [0 , +\infty [$ are twice continuously differentiable nonnegative functions, besides $\psi _{2} (s)  $ is increasing function; $C$ is linear operator, which satisfies the following condition
\begin{equation}\left \Vert C u\right \Vert  \leq a_{0} \left \Vert A u\right \Vert \; ,\text{\quad } \forall u \in D (A) \subset D (C) ,\text{\quad }a_{0} =c o n s t >0 ;\text{\quad } \label{E1.3}
\end{equation}
$N$ is linear bounded operator; nonlinear operator $M ( \cdot )$ satisfies Lipschitz condition; $a_{1}$ and $a_{2}$ are positive constants $\varphi _{0}$ and $\varphi _{1}$ are given vectors from $H$; $u \left (t\right )$ is a twice continuously differentiable, unknown function with values in $H$ and $f \left (t\right )$ is given continuous function with values in $H$.

As in the linear case (see S. G. Krein \cite{Kr}) vector function $u \left (t\right )$ with values in $H$, defined on the interval $\left [0 ,\overline{t}\right ]$ is called a solution of the problem (\ref{E1.1})-(\ref{E1.2}) if it satisfies the following conditions: (a) $u \left (t\right )$ is twice continuously differentiable in the interval $\left [0 ,\overline{t}\right ]$; (b) $u \left (t\right ) ,u^{ \prime } \left (t\right ) \in D \left (B\right )$ for any $t$ from $\left [0 ,\overline{t}\right ]$ and $\;B u \left (t\right )$ and $B u^{ \prime } \left (t\right )$ are continuous functions; (c) $u \left (t\right )$ satisfies equation (\ref{E1.1}) on the $\left [0 ,\overline{t}\right ]$ interval and the initial condition (\ref{E1.2}). Here continuity and differentiability is meant by metric $H$.

Equation (\ref{E1.1}) is an abstract analogue of the following equation
\begin{align}\label{B1}
	&\frac{ \partial ^{2}u}{ \partial t^{2}} +a_{1} \frac{ \partial ^{4}}{ \partial x^{4}}\genfrac{(}{)}{}{}{ \partial u}{ \partial t} +a_{2} \frac{ \partial ^{4}u}{ \partial x^{4}} -\left (\alpha  +\beta  \int \limits _{0}^{l}\left [ \partial _{\xi }u \left (\xi  ,t\right )\right ]^{2} d \xi \right ) \frac{ \partial ^{2}u}{ \partial x^{2}} \nonumber\\
	&-\gamma  \frac{ \partial }{ \partial t}\left (\int \limits _{0}^{l}\left [ \partial _{\xi }u \left (\xi  ,t\right )\right ]^{2} d \xi \right ) \frac{ \partial ^{2}u}{ \partial x^{2}} +\delta  \frac{ \partial u}{ \partial t} =f (t) ,\text{\quad }(x ,t) \in ]0 ,l [ \times ] 0 ,\overline{t}]\,,
\end{align}
where $a_{1}$, $a_{2}$, $\beta$ and $\gamma $ are positive and $\alpha $ , $\delta $ any constants.

At the first time, equation (\ref{B1}) was considered by J. M. Ball in \cite{BL}. In this paper, J. M. Ball investigated the existence, uniqueness and asymptotic behaviour of the solution of equation (\ref{B1}) using the topological method.

We look for an approximate solution of the problem (\ref{E1.1})-(\ref{E1.2}) using the following semi-discrete scheme
\begin{align*} &  & \frac{u_{k +1} -2 u_{k} +u_{k -1}}{\tau ^{2}} +a_{1} B \frac{u_{k +1} -u_{k -1}}{2 \tau } +a_{2} B \frac{u_{k +1} +u_{k -1}}{2} \\
	&  &  +a_{1 ,k} A \frac{u_{k +1} +u_{k -1}}{2} +d_{k} A \frac{u_{k +1} +u_{k -1}}{2} +a_{3 ,k} \frac{u_{k +1} +u_{k -1}}{2}
\end{align*}
\begin{equation} +C u_{k} +N \frac{u_{k +1} -u_{k -1}}{2 \tau } +M \left (u_{k}\right ) =f_{k} , \label{eq1}
\end{equation}
where $f_{k} =f \left (t_{k}\right )$, $~k =1 ,\ldots  ,n -1$,$~t_{k} =k \tau $,$~$ $\tau  =\overline{t}/n~$ $\left (n >1\right )$,
\begin{align*}
	&a_{1 ,k} = \widetilde{\psi }_{1} \left (\gamma _{k -1} ,\gamma _{k +1}\right ) ,\text{\quad }\gamma _{k} =\left \Vert A^{1/2} u_{k}\right \Vert ^{2}\text{,}\\
	&d_{k}  = \frac{\psi _{2} (\gamma _{k +1}) -\psi _{2} (\gamma _{k -1})}{2 \tau }\;\text{,}\quad a_{3 ,k} =  \widetilde{\psi }_{3} \left (\left \Vert u_{k -1}\right \Vert ^{2} ,\left \Vert u_{k +1}\right \Vert ^{2}\right )\,,
\end{align*}
and where function $\widetilde{\psi }_{1} (a ,b) $ (analogously of $\widetilde{\psi }_{3} (a ,b)$) is defined using the following formula
\begin{equation}\widetilde{\psi }_{1} (a ,b) =\frac{1}{b -a} \int \limits _{a}^{b}\psi _{1} (s) d s .\text{\quad } \label{E1.5}
\end{equation}

It is clear that if interval $b -a$ is small enough, then (\ref{E1.5}) formula gives good approximation of $\psi _{1} (s) $ function at $s =(a +b)/2$.

In (\ref{eq1}), nonlinear terms are approximated using integral mean. This approach first was used in \cite{Per} and \cite{RTs2}.

As an approximate solution $u \left (t\right )$ of problem (\ref{E1.1})-(\ref{E1.2}) at point $t_{k} =k \tau $ we declare $u_{k}$, $u \left (t_{k}\right ) \approx u_{k}\text{.}$
\begin{remark}\label{remark1.1}
	From (\ref{E1.2.1}) condition it follows that
	\begin{equation}\left \Vert A u\right \Vert  \leq b_{0} \left \Vert B^{1/2} u\right \Vert  ,\text{\quad } \forall u \in D (B) \subset D \left (A\right ) . \label{E1.5.1}
	\end{equation}
\end{remark}
It is known, that $D (B)$ is a core of $B^{1/2 }$ (see \cite{TK}, p. 354). It means that, for every $u \in D (B^{1/2}) $ there exists sequence $u_{n} \in D (B)$ such that,  $u_{n} \rightarrow u$ and $B^{1/2} u_{n} \rightarrow B^{1/2} u$. From here, according to (\ref{E1.5.1}) it follows that $A u_{n}$ is Cauchy sequence and it is clear that, since $H$ is complete, this sequence is convergent. $u \in D (A)$ and $A u_{n} \rightarrow A u$ as $A$ is closed operator. From here and (\ref{E1.5.1}) it follows that:
\begin{equation}\left \Vert A u\right \Vert  \leq b_{0} \left \Vert B^{1/2} u\right \Vert  ,\text{\quad } \forall u \in D (B^{1/2}) \subset D \left (A\right ) . \label{E1.6}
\end{equation}

\section{Uniform boundedness of solution of discrete problem and difference analogue of the first-order derivative}\label{sec:section2}
The following theorem takes place (below everywhere $c$ denotes positive constant).
\begin{theorem}\label{theorem2.1}
	For discrete problem (\ref{eq1}) the vectors $(u_{k} -u_{k -1})/\tau $ and $B^{1/2} u_{k}$ are uniformly bounded, i.e. there exist constants $c_{1}$ and $c_{2}$ (independent of $n$) such that
	\begin{equation*}\genfrac{\Vert }{\Vert }{}{}{u_{k} -u_{k -1}}{\tau } \leq c_{1}\; ,\text{\quad }\left \Vert B^{1/2} u_{k}\right \Vert  \leq c_{2}\; ,\text{\quad }k =1 ,\ldots  ,n\;\text{.}
	\end{equation*}
\end{theorem}
\begin{proof}
	If we multiply both sides of the equality (\ref{eq1}) on vector $u_{k +1} -u_{k -1} =\left (u_{k +1} -u_{k}\right ) +\left (u_{k} -u_{k -1}\right )$, we obtain
	\begin{align}\label{E2.1}
		&\alpha_{k+1}-\alpha_{k}+2{a}_{1}{\tau}{\left\Vert {B}^{1/2}{\delta}{u}_{k} \right\Vert}^{2}+\frac{1}{2}{\beta}_{k+1}-\frac{1}{2}{\beta}_{k-1}+\frac{1}{2}{a}_{1,k}\left( {\gamma}_{k+1}-{\gamma}_{k-1} \right) \nonumber \\
		&+\frac{1}{2}{d}_{k}\left( {\gamma}_{k+1}-{\gamma}_{k-1} \right)+\frac{1}{2}{a}_{3,k}\left( {\vartheta}_{k+1}-{\vartheta}_{k-1} \right)=\left( {g}_{k},{u}_{k +1}-{u}_{k -1} \right)\,,
	\end{align}
	where
	\begin{align*}
		&\alpha _{k}  = \genfrac{\Vert }{\Vert }{}{}{u_{k} -u_{k -1}}{\tau }^{2} ,\text{\quad }\beta _{k} =\left \Vert B^{1/2} u_{k}\right \Vert ^{2}\text{,} \quad\gamma _{k}  = \left \Vert A^{1/2} u_{k}\right \Vert ^{2} ,\text{\quad }\vartheta _{k} =\left \Vert u_{k}\right \Vert ^{2},\\
		&g_{k} =f_{k} -C u_{k} -N \delta{u}_{k} -M \left (u_{k}\right )\,,\quad\delta{u}_{k}=\frac{u_{k +1} -u_{k -1}}{2 \tau }\,.
	\end{align*}
	According to (\ref{E1.5}) we get:
	\begin{align*}
		a_{1 ,k} \left (\gamma _{k +1} -\gamma _{k -1}\right ) &=\widetilde{\psi }_{1} \left (\gamma _{k -1} ,\gamma _{k +1}\right )\left (\gamma _{k +1} -\gamma _{k -1}\right ) \\
		&=\int \limits _{0}^{\gamma _{k +1}}\psi _{1} (s) d s -\int \limits _{0}^{\gamma _{k -1}}\psi _{1} (s) d s\,,
	\end{align*}
	\begin{equation*}a_{3 ,k} (\vartheta _{k +1} -\vartheta _{k -1}) =\int \limits _{0}^{\vartheta _{k +1}}\psi _{3} (s) d s -\int \limits _{0}^{\vartheta _{k -1}}\psi _{3} (s) d s\,.
	\end{equation*}
	Besides according to the monotonicity of function $\psi _{2} (s)$, the following estimation is valid
	\begin{equation*}d_{k} \left (\gamma _{k +1} -\gamma _{k -1}\right ) =\frac{1}{2 \tau } \left (\psi _{2} (\gamma _{k +1}) -\psi _{2} (\gamma _{k -1})\right ) \left (\gamma _{k +1} -\gamma _{k -1}\right ) \geq 0\,.
	\end{equation*}
	Then from (\ref{E2.1}) we get
	\begin{align}
		{\lambda}_{k+1} \leq  {\lambda}_{k} +\left \vert (g_{k} ,u_{k +1} -u_{k -1})\right \vert  , \label{E2.2}
	\end{align}
	where $\lambda _{k}=\alpha _{k} +\frac{1}{2} \left (\beta _{k} +\beta _{k -1} +\mu _{k} +\mu _{k -1} +\nu _{k} +\nu _{k -1}\right )$,
	\begin{equation*}\mu _{k} =\int \limits _{0}^{\gamma _{k}}\psi _{1} (s) d s ,\text{\quad }\nu _{k} =\int \limits _{0}^{\vartheta _{k}}\psi _{3} (s) d s\text{.}
	\end{equation*}
	If we use Schwarz inequality, condition (\ref{E1.3}) and Remark \ref{remark1.1} we obtain
	\begin{align*}
		&\left \vert (g_{k} ,u_{k +1} -u_{k -1})\right \vert  \\
		&\leq  \left \Vert u_{k +1} -u_{k -1}\right \Vert \left (\left \Vert f_{k}\right \Vert  +a_{0} b_{0} \sqrt{\beta_{k}} +\left \Vert M \left (u_{k}\right )\right \Vert  +\left \Vert N \delta{u}_{k}\right \Vert \right )\text{.}
	\end{align*}
	From here follows
	\begin{align}
		\left \vert (g_{k} ,u_{k +1} -u_{k -1})\right \vert  &\leq  \tau  \left (\sqrt{\alpha _{k +1}} +\sqrt{\alpha _{k}}\right )\left (\left \Vert f_{k}\right \Vert \, +\left \Vert M \left (u_{k}\right )\right \Vert \right . \nonumber  \\
		&\left . +c_{1} \sqrt{\beta _{k}} +c_{0} \left (\sqrt{\alpha _{k +1}} +\sqrt{\alpha _{k}}\right )\right ) , \label{E2.3}
	\end{align}
	where $c_{0} =\frac{1}{2} \left \Vert N\right \Vert $, $c_{1} =a_{0} b_{0}$ .
	
	For nonlinear operator $M \left( \cdot \right)$ due to Lipschitz condition we have
	\begin{align}
		\left \Vert M \left (u_{k}\right )\right \Vert  &\leq c \left (\left \Vert u_{k} -u_{k -1}\right \Vert  +\left \Vert u_{k -1} -u_{k -2}\right \Vert  +\ldots  +\left \Vert u_{1} -u_{0}\right \Vert \right ) +\left \Vert M \left (u_{0}\right )\right \Vert  \nonumber \\
		&= c \tau  \left (\sqrt{\alpha _{k}} +\sqrt{\alpha _{k -1}} +\ldots  +\sqrt{\alpha _{1}}\right ) +\left \Vert M \left (u_{0}\right )\right \Vert\,. \label{E2.4}
	\end{align}
	If we insert inequality (\ref{E2.4}) into (\ref{E2.3}) we get
	\begin{align}\label{E2.5}
		\left \vert (g_{k} ,u_{k +1} -u_{k -1})\right \vert \leq \tau  \left (\sqrt{\alpha _{k +1}} +\sqrt{\alpha _{k}}\right ){\sigma}_{k}\,,
	\end{align}
	where
	\begin{equation*}\sigma _{k} = c \tau  \sum \limits _{i =1}^{k}\sqrt{\alpha _{i}} +c_{1} \sqrt{\beta _{k}} +c_{0} \left (\sqrt{\alpha _{k +1}} +\sqrt{\alpha _{k}}\right ) +\left \Vert f_{k}\right \Vert \, +\left \Vert M \left (u_{0}\right )\right \Vert\,.
	\end{equation*}
	From (\ref{E2.2}) according to (\ref{E2.5}) it follows
	\begin{equation}
		\lambda _{k +1} \leq \lambda _{k} +\varepsilon _{k} , \label{eq2}
	\end{equation}
	where $\varepsilon _{k}  = \tau  \left (\sqrt{\alpha _{k +1}} +\sqrt{\alpha _{k}}\right ) \sigma _{k}$.
	
	Obviously from (\ref{eq2}) we obtain
	\begin{align*}
		\lambda _{k +1} \leq \lambda _{1} +\tau  \sum \limits _{i =1}^{k}\left (\sqrt{\alpha _{i}} +\sqrt{\alpha _{i +1}}\right ) \sigma _{i}\text{,}
	\end{align*}
	from here we have
	\begin{equation*}
		\delta _{k +1}^{2} \leq \delta _{1}^{2} +\tau  \sum \limits _{i =1}^{k}\left (\delta _{i} +\delta _{i +1}\right ) \sigma _{i} ,\text{\quad \quad }\delta _{k} =\sqrt{\lambda _{k}}\,\text{.}
	\end{equation*}
	from here follows the following inequality
	\begin{equation}
		\delta _{k +1} \leq \delta _{1} +2 \tau  \sum \limits _{i =1}^{k}\sigma _{i} \leq c \tau  \sum \limits _{i =1}^{k}\delta _{i} +2 c_{0} \tau  \delta _{k +1} +\eta _{k} , \label{E2.6}
	\end{equation}
	where
	\begin{equation*}
		\eta _{k} =\delta _{1} +c \left \Vert M \left (u_{0}\right )\right \Vert  +2 \tau  \sum \limits _{i =1}^{k}\left \Vert f_{i}\right \Vert \text{.}
	\end{equation*}
	If we assume that $1 -2 \tau  c_{0} =1 -\tau  \left \Vert N\right \Vert  >0$, then from (\ref{E2.6}) we have
	\begin{equation*}\delta _{k +1} \leq c \tau  \sum \limits _{i =1}^{k}\delta _{i} +c \eta _{k}\text{.}
	\end{equation*}
	From here, according to discrete analogue of Gr\"{o}nwall's lemma, we have
	\begin{equation*}
		\delta _{k +1} \leq c \exp  (c t_{k -1}) \left (\eta _{k} +\tau  \delta _{1}\right )\text{.}
	\end{equation*}
	From here it follows that $\alpha _{k}$ and $\beta _{k}$ are uniformly bounded.
\end{proof}
\begin{remark}\label{remark2.2}
	From uniform boundedness of vectors $B^{1/2} u_{k}$ follows uniform boundedness of  $A u_{k}$ vectors (see inequality (\ref{E1.6})). From this fact and the following inequality
	\begin{align}\label{eq:sqrt_A}
		\left \Vert A^{1/2} u\right \Vert  = \left \Vert A^{ -1/2} (A u)\right \Vert  \leq \frac{1}{\sqrt{m_{A}}} \left \Vert A u\right \Vert  ,\text{\quad } \forall u \in D (A)\,,
	\end{align}
	where $m_{A} >0$ is lower bound of operator $A$ ($(A u ,u) \geq m_{A} (u ,u)$), follows that $A^{1/2}{u}_{k}$ is uniformly bounded.
\end{remark}
\begin{remark}
	\label{remark2.3}
	According to the triangle inequality, from uniform boundedness of $(u_{k +1} -u_{k})/\tau $ vectors follows uniform boundenss of  $(u_{k +1} -u_{k -1})/2 \tau $ vectors.
\end{remark}
\begin{remark}
	\label{remark2.4}
	From uniform boundenss of $\left \Vert A u_{k}\right \Vert $and $\left \Vert (u_{k +1} -u_{k})/\tau \right \Vert$ follows uniform boundess of
	\begin{equation*}
		\left \vert (\gamma _{k} -\gamma _{k -1})/\tau \right \vert\,.
	\end{equation*}
\end{remark}

\section{Estimations for two-variable Chebyshev polynomials}\label{sec:section3}
To obtain \textit{a priori} estimations for the main linear part of the difference equation (\ref{eq1}) we require estimations for a specific class of polynomials that we call two-variable Chebyshev polynomials. These polynomials are defined using the following recurrence relation (see \cite{R2}):
\begin{align}
	U_{k +1} (x ,y) = x U_{k} (x ,y) -y U_{k -1} (x ,y) ,\text{\quad }k =1 ,2 ,\ldots \text{\quad } , \label{CH1} \\
	U_{1} (x ,y) = x ,\text{\quad }U_{0} (x ,y) \equiv 1. \nonumber
\end{align}
$U_{k} (x ,y)$ we call two-variable Chebyshev polynomials, as $U_{k} \left (2 x ,1\right )$ represents Chebyshev polynomials of the second kind (see, \textit{e.g.} \cite{Sege}).

From recurrence relation (\ref{CH1}) using induction, we get the following formula
\begin{equation}U_{k} (x ,y) =\sqrt{y^{k}} U_{k} (\xi  ,1) ,\text{\quad }\xi  =\frac{x}{\sqrt{y}} ,\text{\quad }y >0. \label{CH2}
\end{equation}
The formula (\ref{CH2}) is important as it relates $U_{k} (x ,y)$ polynomials with classic Chebyshev polynomials (we assume that in classic Chebyshev polynomials $x$ variable is replaced by $x/2$). Let us introduce the following domains:
\begin{align*}
	\Delta  = \left \{(x ,y) :\;\left \vert y\right \vert  <1\,\,\text{and}\,\,\left \vert x\right \vert  <y +1\right \}\text{.} \\
	\Omega ^{ +} = \left \{(x ,y) :\;4 y -x^{2} >0\right \} ,\text{\quad }\Omega ^{ -} =\left \{(x ,y) :\;4 y -x^{2} <0\right \}\text{,} \\
	\Delta ^{ +}  = \left \{(x ,y) \in  \Delta  :\;x \geq 0\right \} ,\text{\quad }\Omega _{1} =\Omega ^{ +} \cap  \Delta ^{ +} ,\text{\quad }\Omega _{2} =\Omega ^{ -} \cap  \Delta ^{ +}\text{.}
\end{align*}

It is well-known that roots of the classic Chebyshev polynomials are in $\;] -1 ,1[$ \  (see, \textit{e.g.}, \cite{Sege}). From here, according to formula (\ref{CH2}), it follows that, for any fixed positive $y$ roots of the polynomial $U_{k} (x ,y)$ are inside $] -2 \sqrt{y} ,2 \sqrt{y}[
$. Besides if we take into consideration, that $U_{k} ( \pm 2 ,1) =( -1)^{k} (k +1) $ and $\left \vert U_{k} (2 \xi  ,1)\right \vert $ reaches its maximum on boundary (see, \textit{e.g.}, \cite{Sege}), then from formula (\ref{CH2}) follows the following estimation
\begin{equation}\left \vert U_{k} (x ,y)\right \vert  \leq U_{k} (2 \sqrt{y} ,y) =(k +1) \sqrt{y^{k}} ,\text{\quad }(x ,y) \in \Omega ^{ +} . \label{CH2.1}
\end{equation}
From above discussion, we conclude that for any positive  $y$, $U_{k} (x ,y)$ is increasing function regarding $x$ variable, when $x \geq 2 \sqrt{y}$. Besides from recurrence relation (\ref{CH1}) it follows that, for any fixed $y \leq 0$, $U_{k} (x ,y)$ is increasing function regarding to $x$ variable, when $x \geq 0$. From here we obtain
\begin{equation}\left \vert U_{k} (x ,y)\right \vert  \leq U_{k} (1 +y ,y) =1 +y +\ldots  +y^{k} , \label{CH3}
\end{equation}
where $y \geq  -1$and $\left \vert x\right \vert  \leq 1 +y$.

From (\ref{CH3}) follows the following estimation
\begin{equation}
	\left \vert U_{k} (x ,y) (1 -y)\right \vert  \leq 1 ,\text{\quad }(x ,y) \in  \Delta  . \label{CH3.1}
\end{equation}
We also need estimation for $U_{k} (x ,y) -y^{m} U_{k -1} (x ,y)$, $m =0 ,1$, polynomials, where $(x ,y) \in  \Delta ^{ +}$.

The following inequality is simply obtained
\begin{equation}
	\left \vert U_{k} (x ,1) -U_{k -1} (x ,1)\right \vert  \leq \frac{2}{\sqrt{2 +x}} ,\text{\quad }x \in ] -2 ,2] . \label{CH5}
\end{equation}
According to formula (\ref{CH2}) and inequality (\ref{CH5}) the following estimation is valid
\begin{align}
	\left \vert U_{k} (x ,y) -\sqrt{y} U_{k -1} (x ,y)\right \vert = \sqrt{y^{k}} \left \vert U_{k} (\xi  ,1) -U_{k -1} (\xi  ,1)\right \vert  \leq \sqrt{2 y^{k}}\; , \label{CH6}
\end{align}
where $\xi  =x/\sqrt{y}$ , \ $(x ,y) \in \Omega _{1}$.

Let us estimate the following difference $U_{k} (x ,y) -y U_{k -1} (x ,y)$, when $(x ,y) \in \Omega _{1}$. According to inequalities (\ref{CH6}) and (\ref{CH2.1}) we have
\begin{align}
	\left \vert U_{k} (x ,y) -y U_{k -1} (x ,y)\right \vert \leq \sqrt{2} ,\text{\quad }(x ,y) \in \Omega _{1} . \label{CH7}
\end{align}

Let us estimate the following difference $U_{k} (x ,y) -y U_{k -1} (x ,y)$, when $(x ,y) \in \Omega _{2}$ and $y >0$. For that we require the following formulas:
\begin{equation}
	U_{k} (x ,y) =\sqrt{y^{k}} \sum _{i =0}^{k}C_{k +i +1}^{2 i +1} (\xi  -2)^{i} , \label{CH7.1}
\end{equation}
\begin{equation}
	U_{k} (x ,y) -\sqrt{y} U_{k -1} (x ,y) =\sqrt{y^{k}} \sum _{i =0}^{k}C_{k +i}^{2 i} (\xi  -2)^{i} , \label{CH7.2}
\end{equation}
where $\xi  =x/\sqrt{y}$ , $\;C_{k}^{i}$ are binomial coefficients ($C_{k}^{0} =1$).

Using simple transformation from (\ref{CH7.1}) we obtain formula (\ref{CH7.2}). The formula (\ref{CH7.1}) can be obtained by using Taylor expansion of $U_{k} (\xi  ,1) $ at $\xi  =2$. We need to consider that $U_{k}^{(i)} (2 ,1) =i ! C_{k +i +1}^{2 i +1}$.

From the following equality
\begin{align*}
	U_{k} (x ,y) -y U_{k -1} (x ,y) = \left (U_{k} (x ,y) -\sqrt{y} U_{k -1} (x ,y)\right ) +\left (1 -\sqrt{y}\right ) \sqrt{y} U_{k -1} (x ,y)\text{,}
\end{align*}
according to formulas (\ref{CH7.1}) and (\ref{CH7.2}), it follows that, for any fixed $y$ from $]0 ,1] $ interval $U_{k} (x ,y) -y U_{k -1} (x ,y)$ is increasing function regarding $x$, when $x \geq 2 \sqrt{y}$. From here follows
\begin{align}
	U_{k} \left (2 \sqrt{y} ,y\right ) -y U_{k -1} \left (2 \sqrt{y} ,y\right ) \leq  U_{k} -y U_{k -1} \leq U_{k} (1 +y ,y) -y U_{k -1} (1 +y ,y)\,, \label{CH8}
\end{align}
where $y >0$ and $(x ,y) \in \Omega _{2}$.

If we insert $U_{k} (2 \sqrt{y} ,y) =(k +1) \sqrt{y^{k} }$ and \eqref{CH3} in the relation (\ref{CH8}) then we obtain the following estimation
\begin{equation}\sqrt{y^{k}} \left ((k +1) \left (1 -\sqrt{y}\right ) +\sqrt{y}\right ) \leq U_{k} (x ,y) -y U_{k -1} (x ,y) \leq 1 , \label{CH9}
\end{equation}
where $y >0$ and $(x ,y) \in \Omega _{2}$.

We straightforwardly obtain the following inequality
\begin{equation}
	0 \leq U_{k} (x ,y) -y U_{k -1} (x ,y) \leq 1 , \label{CH10}
\end{equation}
where $y \leq 0$ and $(x ,y) \in \Omega _{2}$.

From estimations (\ref{CH7}), (\ref{CH9}) and (\ref{CH10}) we have
\begin{equation}
	\left \vert U_{k} (x ,y) -y U_{k -1} (x ,y)\right \vert  \leq \sqrt{2} ,\text{\quad }(x ,y) \in  \Delta ^{ +} . \label{CH11}
\end{equation}
Analogously to \eqref{CH11} we obtain
\begin{equation}
	\left \vert U_{k} (x ,y) -U_{k -1} (x ,y)\right \vert  \leq \sqrt{2} ,\text{\quad }(x ,y) \in  \Delta ^{ +} . \label{CH17}
\end{equation}

\section{High Order \textit{a priori} estimations for three-layer semi-discrete scheme corresponding to the second-order evolution equation}\label{sec:section4}
The three-layer scheme is natural for the second-order evolution equation and the two-layer scheme is natural for the first-order evolution equation. Investigation of a three-layer scheme is more difficult than an investigation of a two-layer scheme. This difficulty can be somehow simplified if we reduce the second-order evolution equation to the first-order evolution equation by introducing additional unknowns. In this case, a self-adjoint operator is replaced by an operator matrix that is not self-adjoint anymore. This makes it complicated to investigate the corresponding discrete problem.

Obtaining such \textit{a priori} estimations from where follows stability and convergence of the nonlinear semi-discrete scheme (\ref{eq1}), is based on high order accuracy \textit{a priori} estimations for corresponding linearized semi-discrete scheme. In this section, we obtain \textit{a priori} estimation for a three-layer semi-discrete scheme for the second-order evolution equation. In this estimation on the left-hand side, we have positive $s$ power for the main operator, while on the right-hand side we have $(s -1) $. This allows us to make weaken nonlinear part of the equation in such a way that, using the results obtained in the previous section we are able to use Gr\"{o}nwall's lemma and obtain the final estimation. To obtain these estimations we require to construct the exact representation of three-layer recurrence relations with operator coefficients by using two-variable Chebyshev polynomials. These kinds of estimates were obtained before by one of the authors of the presented papers (see \cite{R2}, \cite{RJ}).

Important results for constructing and investigating approximate schemes for the Cauchy problem for second-order evolution equations were obtained by the following authors: G. A. Baker \cite{Bak1}, G. A. Baker, J. H. Bramble \cite{Bak2}, G. A. Baker, V. A. Dougalis, S. M. Serbin \cite{Bak3}, L. A. Bales \cite{Bal}, J. Ka\v{c}ur \cite{Kac}, O. Ladyzhenskaya \cite{Lad}, M. Pultar \cite{Pul}, P. E. Sobolevskij, L. M. Chebotarova \cite{Sob}.

Let us consider in Hilbert space $H$ the following linear difference  equation
\begin{equation}
	\frac{u_{k +1} -2 u_{k} +u_{k -1}}{\tau ^{2}} +a_{1} B \frac{u_{k +1} -u_{k -1}}{2 \tau } +a_{2} B \frac{u_{k +1} +u_{k -1}}{2} =f_{k} , \label{uklin}
\end{equation}
where $k =1 ,\ldots  ,n -1$, $u_{0}$, $u_{1}$ and $f_{k}$ are the given vectors from $H$.

Difference equation (\ref{uklin}) represents main part of the nonlinear equation (\ref{eq1}), and obviously corresponds main part of the equation (\ref{E1.1}).

The following lemma takes place.
\begin{lemma}\label{lemma4.1}
	Let $B$ be self-adjoint positively defined operator and  $1 -\tau  a >0$, $a =a_{2}/a_{1}$. Then for scheme (\ref{uklin}) the following \textit{a priori} estimation is valid:
\begin{align}
	\left \Vert B^{s} u_{k +1}\right \Vert  \leq  \sqrt{2} \left \Vert B^{s} u_{0}\right \Vert  +\frac{1}{a_{1}} \left \Vert B^{s -1} \frac{ \Delta u_{0}}{\tau }\right \Vert  +\frac{\tau }{2} (1 +\tau  a) \left \Vert B^{s} \frac{ \Delta u_{0}}{\tau }\right \Vert  \nonumber  \\
	\text{\quad }  + \frac{\tau }{a_{1}} \sum \limits _{i =1}^{k}\left \Vert B^{s -1} f_{i}\right \Vert  ,\text{\quad }u_{0} ,u_{1} \in D (B^{s}) ,~f_{i} \in D (B^{s -1}) ,~ \label{u1} \\
	\genfrac{\Vert }{\Vert }{}{}{ \Delta u_{k}}{\tau }  \leq  a \left \Vert u_{0}\right \Vert  +\sqrt{2} \genfrac{\Vert }{\Vert }{}{}{ \Delta u_{0}}{\tau } +\sqrt{2} \tau  \sum \limits _{i =1}^{k}\left \Vert f_{i}\right \Vert  , \label{u2}\end{align}where $s \geq 0~$, $k =1 ,\ldots  ,n -1\text{,}$ $ \Delta u_{k} =u_{k +1} -u_{k}$ $(B^{0} =I)$.
\end{lemma}
\begin{remark}
	\label{remark4.2}
	In section \ref{sec:section2} (see Theorem \ref{theorem2.1} and Remark \ref{remark2.2}, Remark \ref{remark2.3} and Remark \ref{remark2.4}) there was shown that $\left \Vert (u_{k +1} -u_{k})/\tau \right \Vert $, $\left \Vert A^{1/2} u_{k}\right \Vert $, $\left \Vert A u_{k}\right \Vert $, $\left \Vert B^{1/2} u_{k}\right \Vert $ and $\left \vert (\gamma _{k} -\gamma _{k -1})/\tau \right \vert $ are uniformly bounded. These results along with (\ref{u1}) and (\ref{u2}) allows to obtain such \textit{a priori} estimations for semi-discrete scheme (\ref{eq1}) from where follows stability and convergence of the presented method.
\end{remark}
\begin{remark}
	\label{remark4.3}
	\textit{A priori} estimations (\ref{u1}) and (\ref{u2}) have independent meaning, as constants on the right-hand side are absolute constants (does not depend on interval length). Besides these constants cannot be improved. To obtain these estimations were possible by constructing exact representations for the solution of difference equation (\ref{uklin}).
\end{remark}
\begin{proof}[Proof of Lemma \ref{lemma4.1}]
	Let us rewrite difference equation (\ref{uklin}) in the following form to prove inequality (\ref{u1})
	\begin{equation}\label{eq:three-layer_equation}
		{B}_{0}u_{k +1} -2 I u_{k} +{B}_{1}u_{k -1} =\tau ^{2} f_{k}\,,
	\end{equation}
	where
	\begin{equation*}
		{B}_{0}=I +\frac{\tau }{2} a_{1} B +\frac{\tau ^{2}}{2} a_{2} B\,,\quad{B}_{1} = {B}_{0} - \tau{a}_{1}{B}\,.
	\end{equation*}
	From \eqref{eq:three-layer_equation} we get
	\begin{equation}
		u_{k +1} =L u_{k} -S u_{k -1} +\frac{\tau ^{2}}{2} L f_{k} , \label{rek}
	\end{equation}
	where $L = 2{B}_{0}^{-1}$, $S = {B}_{1}{B}_{0}^{-1}$.
	
	Let us note that $L$ and $S$ are self-adjoint, bounded linear operators in a Hilbert space $H$. Besides it is obvious that $L$ and $S$ are commutative.
	
	Using mathematical induction for (\ref{rek}) recurrence relation we get
	\begin{equation}
		u_{k +1} =U_{k} (L ,S) u_{1} -S U_{k -1} (L ,S) u_{0} +\frac{\tau ^{2}}{2} \sum \limits _{i =1}^{k}U_{k -i} (L ,S) L f_{i} , \label{uk+1}
	\end{equation}
	where operator polynomial $U_{k} (L ,S)$ satisfies the following recurrence relation:
	\begin{align}
		& U_{k} (L ,S)  = L U_{k -1} (L ,S) -S U_{k -2} (L ,S) ,\text{\quad }k =1 ,2 ,\ldots \text{\quad } , \label{UKL} \\
		& U_{0} (L ,S) = I ,\text{\quad }U_{ -1} (L ,S) =0. \nonumber
	\end{align}
    Scalar polynomials $U_{k} (x ,y) $ corresponding to $U_{k} (L ,S)$ satisfy (\ref{CH1}) recurrence relation.
	
	From (\ref{uk+1}) using simple transformation we have
	\begin{align}
		& u_{k +1} = \tau  U_{k} (L ,S) \frac{ \Delta u_{0}}{\tau } +\left (U_{k} (L ,S) -S U_{k -1} (L ,S)\right ) u_{0} +\frac{\tau ^{2}}{2} \sum \limits _{i =1}^{k}U_{k -i} (L ,S) L f_{i} . \label{uk.1}
	\end{align}
	If we apply operator $B^{s}$ ($s \geq 0$) on both sides of equality (\ref{uk.1}) and move on to the norm we obtain
	\begin{align}
		& \left \Vert B^{s} u_{k +1}\right \Vert  \leq \tau  \left \Vert B^{s} U_{k} (L ,S) \frac{ \Delta u_{0}}{\tau }\right \Vert  +\left \Vert U_{k} (L ,S) -S U_{k -1} (L ,S)\right \Vert  \left \Vert B^{s} u_{0}\right \Vert  \nonumber  \\
		&  +\frac{\tau ^{2}}{2} \sum \limits _{i =1}^{k}\left \Vert B^{s} U_{k -i} (L ,S) L f_{i}\right \Vert  . \label{E3.1}
	\end{align}
	Using simple transformations we get:
	\begin{equation}\label{E3.2}
		S  = \frac{1}{1 +\tau  a} L -\frac{1 -\tau  a}{1 +\tau  a} I\,,
	\end{equation}
	\begin{equation}\label{ELS}
		I-S =\frac{\tau  a_{1}}{2} B L\,.
	\end{equation}
	From here we have
	\begin{align}
		& \tau  \left \Vert B^{s} U_{k} (L ,S) L f\right \Vert  = \tau  \left \Vert U_{k} (L ,S) B L (B^{s -1} f)\right \Vert  \nonumber  \\
		&  = \frac{2}{a_{1}} \left \Vert U_{k} (L ,S) (I -S) (B^{s -1} f)\right \Vert  \leq  \frac{2}{a_{1}} \left \Vert U_{k} (L ,S) (I -S)\right \Vert  \left \Vert B^{s -1} f\right \Vert  , \label{E3.3}
	\end{align}
	where $f \in D (B^{s -1})$.
	
	As it is known the norm of an operator function, when the argument of the function represents self-adjoint bounded operator, is equal to $C$-norm of the corresponded scalar function (see, \textit{e.g.}, \cite{Reed}, Chapter VII). Using this and representation (\ref{E3.2}) we have
	\begin{align}
	    & \left \Vert U_{k} (L ,S) (I -S)\right \Vert  \leq \max_{y \in \sigma  (S)}\left \vert U_{k} (\eta  (y) ,\text{\/}y)(1 -y)\right \vert  , \label{E3.4}
    \end{align}
    where $\eta  (y) =(1 +\tau  a) y +(1 -\tau  a)$.
    
    Let us estimate the spectrum of operator $L$. As, according to the condition, $B$ is a self-adjoint and positively defined operator, therefore, we have $\sigma  (L) \subset \left[ 0 ,2 \right]$. According to this relation, from the representation (\ref{E3.2}) we have
    \begin{equation}
    	\sigma  (S) \subset [ -1 ,1] . \label{u15.1}
    \end{equation}
    If we consider relation (\ref{u15.1}) and estimation (\ref{CH3.1}) we get
    \begin{align*}
    	& \max_{y \in \sigma  (S)}\left \vert U_{k} (\eta  (y) ,\text{\/}y)(1 -y)\right \vert  \leq  \max_{y \in [ -1 ,1]}\left \vert U_{k} (\eta  (y) ,\text{\/}y)(1 -y)\right \vert  \\
    	&  \leq \max_{(x ,y) \in  \Delta ^{ +}}\left \vert U_{k} (x ,\text{\/}y)(1 -y)\right \vert  \leq 1.
    \end{align*}
    From here and (\ref{E3.4}) it follows
    \begin{equation}
    	\left \Vert U_{k} (L ,S) (I -S)\right \Vert  \leq 1. \label{u17}
    \end{equation}
    It is obvious, from (\ref{E3.3}) using (\ref{u17}) we have
    \begin{equation}
    	\tau  \left \Vert B^{s} U_{k} (L ,S) L f\right \Vert  \leq \frac{2}{a_{1}} \left \Vert B^{s -1} f\right \Vert  ,\text{\quad }f \in D (B^{s -1}) . \label{u18}
    \end{equation}
    Also from (\ref{u18}) follows the inequality
    \begin{align}
    	& \tau  \left \Vert B^{s} U_{k} (L ,S) f\right \Vert   \leq  \frac{1}{a_{1}} \left \Vert B^{s -1} {B}_{0} f\right \Vert  \leq \frac{1}{a_{1}} \left \Vert B^{s -1} f\right \Vert  +\frac{\tau }{2} (1 +\tau  a) \left \Vert B^{s} f\right \Vert  , \label{u19}
    \end{align}
    where $f \in D (B^{s})$.
    
    Let us estimate operator $U_{k} -S U_{k -1}$. Analogously, we have
    \begin{equation*}
    	\left \Vert U_{k} (L ,S) -S U_{k -1} (L ,S)\right \Vert  \leq \max_{y \in \sigma  (S)}\left \vert U_{k} (\eta  (y),y) -y U_{k -1} (\eta  (y),y)\right \vert \text{.}
    \end{equation*}
    If we consider relation (\ref{u15.1}) and estimation (\ref{CH11}) we get
    \begin{align*}
    	& \max_{y \in \sigma  (S)}\left \vert U_{k} (\eta  (y) ,\text{\/}y) -y U_{k -1} (\eta  (y) ,\text{\/}y)\right \vert  \leq \max_{(x ,y) \in  \Delta ^{ +}}\left \vert U_{k} (x ,\text{\/}y)1 -y U_{k -1} (x ,\text{\/}y)\right \vert  \leq \sqrt{2}\text{.}
    \end{align*}
    So, we have
    \begin{equation}
    	\left \Vert U_{k} (L ,S) -S U_{k -1} (L ,S)\right \Vert  \leq \sqrt{2} . \label{u20}
    \end{equation}
    If we insert estimations (\ref{u18}), (\ref{u19}) and (\ref{u20}) into inequality (\ref{E3.1}) we obtain estimation (\ref{u1}).
    
    Let us prove estimation (\ref{u2}). From the formula \eqref{uk.1} according to the recurrence relation \eqref{UKL} we get
    \begin{align}
    	& \genfrac{\Vert }{\Vert }{}{}{ \Delta u_{k}}{\tau }  \leq  \tau ^{ -1} \left \Vert (L -S -I) U_{k -1} (L ,S)\right \Vert  \left \Vert u_{0}\right \Vert  \nonumber  \\
    	& +\left \Vert U_{k} -U_{k -1}\right \Vert  \genfrac{\Vert }{\Vert }{}{}{ \Delta u_{0}}{\tau } +\frac{\tau }{2} \sum \limits _{i =1}^{k}\left \Vert U_{k -i} -U_{k -i -1}\right \Vert  \left \Vert L\right \Vert  \left \Vert f_{i}\right \Vert  . \label{u30}
    \end{align}
    According to equality (\ref{ELS}) we have
    \begin{align*}
    	L -S -I  = I -S -(2 I -L) = I -S -(1 +\tau  a) (I -S) = -\tau  a (I -S)\text{.}
    \end{align*}
    If we consider this representation and estimation (\ref{u17}), we get
    \begin{equation}
    	\tau ^{ -1} \left \Vert (L -S -I) U_{k -1} (L ,S)\right \Vert  =a \left \Vert U_{k} (L ,S) (I -S)\right \Vert  \leq a . \label{u31}
    \end{equation}
    
    Analogously to (\ref{u20}), according to estimation (\ref{CH17}) we obtain
    \begin{equation}
    	\left \Vert U_{k} (L ,S) -U_{k -1} (L ,S)\right \Vert  \leq \sqrt{2} . \label{u32}
    \end{equation}
    From inequality (\ref{u30}) using estimations (\ref{u31}), (\ref{u32}) and $\left \Vert L\right \Vert  \leq 2$ \textit{a priori} estimation (\ref{u2}) follows.
\end{proof}
\begin{remark}
\label{remark4.4}
Analogously to (\ref{u2}) the following estimation can be obtained
\begin{equation}\left \Vert B^{s} \frac{ \Delta u_{k}}{\tau }\right \Vert  \leq a \left \Vert B^{s} u_{0}\right \Vert  +\sqrt{2} \left \Vert B^{s} \frac{ \Delta u_{0}}{\tau }\right \Vert  +\sqrt{2} \tau  \sum \limits _{i =1}^{k}\left \Vert B^{s} f_{i}\right \Vert  , \label{u2.1}
\end{equation}
where $u_{0} ,~u_{1} ,f_{i} \in D (B^{s}) ,\; s \geq 0$.
\end{remark}

\section{The \textit{a priori} estimates for perturbation of the solution of the discrete problem}\label{sec:section5}
The goal of the section is to show the stability of the scheme \eqref{eq1}. As additivity does not take place for nonlinear cases, therefore it is natural that we try to obtain the \textit{a priori} estimates exactly for the solution perturbation. From here (analogously to a linear problem) automatically
follows the stability and convergence of the nonlinear scheme.

In this section, based on the results of the previous sections, we obtain the \textit{a priori} estimates for the solution of the semi-discrete scheme (\ref{eq1}) and perturbation of the corresponding first-order difference.

The following theorem takes place (below everywhere $c$ denotes positive constant).
\begin{theorem}\label{theorem5.1}
	Let $u_{k}$ and $\overline{u}_{k}$ be solutions of difference equation (\ref{eq1}) corresponding to initial vectors $(u_{0} ,\;u_{1} ,\;f_{k})$ and $(\overline{u}_{0} ,\;\overline{u}_{1} ,\;\overline{f}_{k})$, components of which are sufficiently smooth. Then for $z_{k} =u_{k} -\overline{u}_{k}$ the following estimates are true
	\begin{align}
		& \left \Vert B^{1/2} z_{k +1}\right \Vert  +\genfrac{\Vert }{\Vert }{}{}{ \Delta z_{k}}{\tau } \nonumber  \\
		&  \leq  c \left (\left \Vert B^{1/2} z_{0}\right \Vert  +\genfrac{\Vert }{\Vert }{}{}{ \Delta z_{0}}{\tau } +\tau  \left \Vert B^{1/2} \frac{ \Delta z_{0}}{\tau }\right \Vert  +\tau  \sum \limits _{i =1}^{k}\left \Vert f_{i} -\overline{f}_{i}\right \Vert \right )\; , \label{T2_2}\end{align}
	where $k =1 ,\ldots  ,n -1$ , $ \Delta z_{k} =z_{k +1} -z_{k}$.
\end{theorem}

In this section and the following ones the notation $\delta{z}_{k}=\dfrac{z_{k +1} -z_{k -1}}{2\tau}$ is applied.

Let us prove the corresponding auxiliary lemma.
\begin{lemma}\label{lemma5.2}
	The following inequality is true
	\begin{equation}
		\left \vert d_{k} -\overline{d}_{k}\right \vert  \leq c \left (\left \Vert A z_{k +1}\right \Vert  +\left \Vert A z_{k -1}\right \Vert  +\Vert \delta{z}_{k}\right \Vert ) , \label{dk1}
	\end{equation}
	where $z_{k} =u_{k} -\overline{u}_{k}$ ,
	\begin{align*}
		d_{k}  =  \delta{\psi_{2}}\left( \gamma_{k} \right)\,,\quad\overline{d}_{k} = \delta{\psi_{2}}\left( \overline{\gamma}_{k} \right)\,,\quad \gamma _{k} =\left \Vert A^{1/2} u_{k}\right \Vert ^{2}\,,\quad\overline{\gamma }_{k} =\left \Vert A^{1/2} \overline{u}_{k}\right \Vert ^{2}\,.
	\end{align*}
\end{lemma}
\begin{proof}
	The following representation is valid
	\begin{align}
		& d_{k} -\overline{d}_{k}  = \left(\delta\gamma_{k} - \delta\overline{\gamma}_{k} \right) I_{1 ,k} +\delta\overline{\gamma}_{k}I_{2 ,k}\,,\label{dk1.1}
	\end{align}
	where
	\begin{align*}
		I_{1 ,k}  = \int \limits _{0}^{1}{\psi _{2}^{ \prime } \left(l_{k} \left( \xi \right) \right)}\mathrm{d}\xi \,,\quad I_{2 ,k} = \int \limits _{0}^{1}{\left[\psi _{2}^{ \prime } \left(l_{k} \left( \xi \right) \right) -\psi _{2}^{ \prime } \left (\overline{l}_{k} \left( \xi \right) \right ) \right]}\mathrm{d}\xi\,,
	\end{align*}
	and where
	\begin{equation*}
		l_{k} \left( \xi \right)  =  \gamma _{k -1} +\left (\gamma _{k +1} -\gamma _{k -1}\right ) \xi \text{,} \quad\overline{l}_{k} \left( \xi \right) = \overline{\gamma }_{k -1} +\left (\overline{\gamma }_{k +1} -\overline{\gamma }_{k -1}\right ) \xi \text{.}
	\end{equation*}
	Using standard transformations we get
	\begin{align*}
		&\delta\gamma_{k} - \delta\overline{\gamma}_{k} = \left (A z_{k +1} ,\delta{u}_{k} +\delta\overline{u}_{k}\right ) +\left (\delta{z}_{k} ,A (u_{k -1} +\overline{u}_{k -1})\right )\text{.}
	\end{align*}
	From here using the Schwarz inequality follows
	\begin{align}
		& \vert \delta\gamma_{k} - \delta\overline{\gamma}_{k} \vert \leq  \left \Vert A z_{k +1}\right \Vert \left( \Vert \delta{u}_{k} \Vert + \Vert \delta\overline{u}_{k} \Vert \right )+ \Vert \delta{z}_{k} \Vert \left (\left \Vert A u_{k -1}\right \Vert  +\left \Vert A \overline{u}_{k -1}\right \Vert \right ) . \label{dk1.5}
	\end{align}
	If we consider that all terms on the right-hand side of the inequality (\ref{dk1.5}) are bounded (see Remark \ref{remark2.2} and Remark \ref{remark2.3}), then we obtain
	\begin{align}
		& \vert \delta\gamma_{k} - \delta\overline{\gamma}_{k} \vert  \leq   c \left (\left \Vert A z_{k +1}\right \Vert  + \Vert \delta{z}_{k} \Vert \right ) . \label{dk1.6}
	\end{align}
	
	Let us estimate integrals $I_{1 ,k}$ and $I_{2 ,k }$. If we use the change of variables for the integral $I_{2 ,k }$ we obtain
	\begin{equation}
		I_{2 ,k} =\int \limits _{0}^{1}\int \limits _{\overline{l}_{k} (\xi )}^{l_{k} (\xi )}\psi _{2}^{ \prime  \prime } \left (\eta \right ) d \eta  d \xi\,. \label{dk1.7}
	\end{equation}
	From (\ref{dk1.7}) we have
	\begin{equation}
		\left \vert I_{2 ,k}\right \vert  \leq c \int \limits _{0}^{1}\left \vert \chi  (\xi ) \right \vert  d \xi  , \label{dk1.8}
	\end{equation}
	where
	\begin{align*}
		&\chi  (\xi ) =\left (\gamma _{k -1} -\overline{\gamma }_{k -1}\right ) +\left (\left (\gamma _{k +1} -\overline{\gamma }_{k +1}\right ) -\left (\gamma _{k -1} -\overline{\gamma }_{k -1}\right )\right ) \xi\,,\\
		&c =\max \left \vert \psi _{2}^{ \prime  \prime } \left (\eta \right )\right \vert  < +\infty\,,\quad 0 \leq \eta  \leq \max_{k}(\gamma _{k} ,\overline{\gamma }_{k}) < +\infty\,.
	\end{align*}
	
	Regarding inequality (\ref{dk1.8}) it should be noted that since $\gamma _{k}$ and $\overline{\gamma }_{k}$ are uniformly bounded (see Remark \ref{remark2.2}) therefore $\max_{k}(\gamma _{k} ,\overline{\gamma }_{k}) < +\infty $,  i.e. the interval where we have to find a maximum value of $\left \vert \psi _{2}^{ \prime  \prime } \left (\eta \right )\right \vert  $ is finite. From here follows that $\max \left \vert \psi _{2}^{ \prime  \prime } \left (\eta \right )\right \vert  < +\infty $ (according to the condition $\psi _{2}^{ \prime  \prime } \left (\eta \right ) $ is continuous). So, from (\ref{dk1.8}) we have
	\begin{align}
		& \left \vert I_{2 ,k}\right \vert  \leq  c \int \limits _{0}^{1}\left \vert \chi  (\xi )\right \vert  d \xi  \leq  c \left (\left \vert \gamma _{k -1} -\overline{\gamma }_{k -1}\right \vert  +\left \vert \gamma _{k +1} -\overline{\gamma }_{k +1}\right \vert \right )\,. \label{dk3}
	\end{align}
	As vectors $A^{1/2} u_{k}$ are uniformly bounded, for the difference $\gamma _{k} -\overline{\gamma }_{k}$ the following estimation is valid
	\begin{align}
		& \left \vert \gamma _{k} -\overline{\gamma }_{k}\right \vert   =  \left (\sqrt{\gamma _{k}} +\sqrt{\overline{\gamma }_{k}}\right ) \left \vert \sqrt{\gamma _{k}} -\sqrt{\overline{\gamma }_{k}}\right \vert \leq  c \left \Vert A^{1/2} z_{k}\right \Vert  \leq c \left \Vert A z_{k}\right \Vert  . \label{gamak}
	\end{align}
	From (\ref{dk3}) according to (\ref{gamak}) follows
	\begin{equation}
		\left \vert I_{2 ,k}\right \vert  \leq c \left (\left \Vert A z_{k +1}\right \Vert  +\left \Vert A z_{k -1}\right \Vert \right ) . \label{dk4}
	\end{equation}
	For the integrals $I_{1 ,k }$ the following estimation is valid
	\begin{equation}
		\left \vert I_{1 ,k}\right \vert  \leq c ,\text{\quad }c =\max \left \vert \psi _{2}^{ \prime } \left (s\right )\right \vert  < +\infty  ,\text{\quad }0 \leq s \leq \max_{k}(\gamma _{k} ,\overline{\gamma }_{k}) < +\infty \; . \label{dk5}
	\end{equation}
	From (\ref{dk1.1}) according to inequalities (\ref{dk1.6}), (\ref{dk4}), (\ref{dk5}) and Remark \ref{remark2.4} follows (\ref{dk1}).
\end{proof}

Let us return to proof of \textbf{Theorem \ref{theorem5.1}.}
\begin{proof}
	According to (\ref{eq1}) difference $z_{k} =u_{k} -\overline{u}_{k}$ satisfies the following equation
	\begin{equation}
		\frac{z_{k +1} -2 z_{k} +z_{k -1}}{\tau ^{2}} +a_{1} B \frac{z_{k +1} +z_{k -1}}{2} +a_{2} B \frac{z_{k +1} +z_{k -1}}{2} = -\frac{1}{2} g_{k} ,\text{\quad } \label{zk1}
	\end{equation}
	where $k =1 ,\ldots  ,n -1\text{,}$
	\begin{align*}
		& g_{k} = g_{1 ,k} +g_{2 ,k} +g_{3 ,k} +g_{4 ,k} +g_{5 ,k}\text{,} \\
		& g_{1 ,k} = a_{1 ,k} A \left (u_{k +1} +u_{k -1}\right ) -\overline{a}_{1 ,k} A \left (\overline{u}_{k +1} +\overline{u}_{k -1}\right )\text{,} \\
		& g_{2 ,k}  = d_{k} A \left (u_{k +1} +u_{k -1}\right ) -\overline{d}_{k} A \left (\overline{u}_{k +1} +\overline{u}_{k -1}\right )\text{,} \\
		& g_{3 ,k} = a_{3 ,k} A \left (u_{k +1} +u_{k -1}\right ) -\overline{a}_{3 ,k} A \left (\overline{u}_{k +1} +\overline{u}_{k -1}\right )\text{,} \\
		& g_{4 ,k} =  2 \left (C z_{k} +N\delta{z}_{k} \right )\,,\quad g_{5 ,k} = 2 \left (\left (M \left (u_{k}\right ) -M \left (\overline{u}_{k}\right )\right ) -(f_{k} -\overline{f}_{k})\right )\text{,}
	\end{align*}
	and where $\overline{a}_{1 ,k} =\widetilde{\psi }_{1} \left (\overline{\gamma }_{k +1} ,\overline{\gamma }_{k -1}\right )\text{,}$ $\overline{\gamma }_{k} =\left \Vert A^{1/2} \overline{u}_{k}\right \Vert ^{2}$ (analogously are defined $\overline{a}_{3 ,k}$ and $\overline{d}_{k}$).
	
	For scheme (\ref{zk1}), according to Lemma \ref{lemma4.1} the following \textit{a priori} estimations are valid (see (\ref{u1}) and (\ref{u2})):
	\begin{align}
		\left \Vert B^{1/2} z_{k +1}\right \Vert &\leq  c\left( \left \Vert B^{1/2} z_{0}\right \Vert  + \left \Vert \frac{ \Delta z_{0}}{\tau }\right \Vert + \tau \left \Vert B^{1/2} \frac{ \Delta z_{0}}{\tau }\right \Vert +\tau \sum \limits _{i =1}^{k}\left \Vert g_{i}\right \Vert \right) \,, \label{zk2} \\
		\genfrac{\Vert }{\Vert }{}{}{ \Delta z_{k}}{\tau } &\leq c\left( \left \Vert z_{0}\right \Vert  + \genfrac{\Vert }{\Vert }{}{}{ \Delta z_{0}}{\tau } + \tau  \sum \limits _{i =1}^{k}\left \Vert g_{i}\right \Vert \right)\,. \label{zk2_1}
	\end{align}
	
	Let us estimate each $g_{ . ,k}$ separately. For $g_{1 ,k }$ we have
	\begin{equation}
		g_{1 ,k} =\left (a_{1 ,k} -\overline{a}_{1 ,k}\right ) \left (A u_{k +1} +A u_{k -1}\right ) +\overline{a}_{1 ,k} \left (A z_{k +1} +A z_{k -1}\right ) . \label{g1_k}
	\end{equation}
	Using simple transformations for difference $a_{1 ,k} -\overline{a}_{1 ,k  }$ we get
	\begin{align}
		a_{1 ,k} -\overline{a}_{1 ,k} = \int \limits _{0}^{1}\int \limits _{\overline{l}_{k} (\xi )}^{l_{k} (\xi )}\psi _{1}^{ \prime } \left (\eta \right ) d \eta  d \xi  , \label{a1_k}
	\end{align}
	From (\ref{a1_k}) according to Remark \ref{remark2.2} we have
	\begin{align}\label{a1_k1}
		&\left \vert a_{1 ,k} -\overline{a}_{1 ,k}\right \vert \leq c \int \limits _{0}^{1}\left \vert \chi  (\xi )\right \vert  d \xi  \leq c \left (\left \vert \gamma _{k -1} -\overline{\gamma }_{k -1}\right \vert  +\left \vert \gamma _{k +1} -\overline{\gamma }_{k +1}\right \vert \right )\,,
	\end{align}
	where $c =\max \left \vert \psi _{1}^{ \prime } \left (\eta \right )\right \vert  < +\infty  ,\text{\quad }0 \leq \eta  \leq \max_{k}(\gamma _{k} ,\overline{\gamma }_{k}) < +\infty \,.$
	
	From (\ref{g1_k}) using (\ref{a1_k1}), we get
	\begin{align}
		\left \Vert g_{1 ,k}\right \Vert &\leq c \left (\left \vert \gamma _{k -1} -\overline{\gamma }_{k -1}\right \vert  +\left \vert \gamma _{k +1} -\overline{\gamma }_{k +1}\right \vert \right )\left (\left \Vert A u_{k +1}\right \Vert  +\left \Vert A u_{k -1}\right \Vert \right ) \nonumber \\
		&+\left \vert \overline{a}_{1 ,k}\right \vert  \left (\left \Vert A z_{k +1}\right \Vert  +\left \Vert A z_{k -1}\right \Vert \right ) . \label{ng1_k}
	\end{align}
	According to Remark \ref{remark2.2}, $\overline{a}_{1 ,k }$ is uniformly bounded. Indeed we have
	\begin{align}
		& \left \vert \overline{a}_{1 ,k}\right \vert  = \left \vert \int \limits _{0}^{1}\psi _{1} \left (\overline{l}_{k}\left( \xi \right) \right ) d \xi \right \vert  \leq c , \label{a_k}
	\end{align}
	where $c =\max \left \vert \psi _{1} \left (s\right )\right \vert  < +\infty  ,\text{\quad }0 \leq s \leq \max_{k}(\gamma _{k} ,\overline{\gamma }_{k})\text{.}$
	
	If we insert estimations (\ref{gamak}) and (\ref{a_k}) in (\ref{ng1_k}) and take into account, that $\left \Vert A u_{k}\right \Vert$ is uniformly bounded, we obtain
	\begin{equation}
		\left \Vert g_{1 ,k}\right \Vert  \leq c \left (\left \Vert A z_{k +1}\right \Vert  +\left \Vert A z_{k -1}\right \Vert \right ) . \label{g1_k1}
	\end{equation}
	
	Let us estimate vector $g_{2 ,k}$. We have
	\begin{equation}
		g_{2 ,k} =\left (d_{k} -\overline{d}_{k}\right ) \left (A u_{k +1} +A u_{k -1}\right ) +\overline{d}_{k} \left (A z_{k +1} +A z_{k -1}\right ) . \label{g2_k}
	\end{equation}
	According to Remark \ref{remark2.2}, for $\overline{d}_{k}$ the following estimation is valid
	\begin{align}
		&  \left \vert \overline{d}_{k}\right \vert  = \left\vert \delta\overline{\gamma}_{k} \right\vert \left \vert \int \limits _{0}^{1}\psi _{2}^{ \prime } \left (\overline{l}_{k}\left( \xi \right) \right ) d \xi \right \vert  \leq  c \left\vert \delta\overline{\gamma}_{k} \right\vert\,, \label{dk_1.1}
	\end{align}
	where $c =\max \left \vert \psi _{2}^{ \prime } \left (s\right )\right \vert  < +\infty  ,\text{\quad }0 \leq s \leq \max_{k}(\gamma _{k} ,\overline{\gamma }_{k})\text{.}$
	
	According to Remark \ref{remark2.3}, from (\ref{dk_1.1}) follows that $\overline{d}_{k}$ is uniformly bounded. Vector $\left \Vert A u_{k}\right \Vert $ is also uniformly bounded (see Remark \ref{remark2.2}). Using this and inequality (\ref{dk1}) from (\ref{g2_k}) follows
	\begin{equation}
		\left \vert g_{2 ,k}\right \vert  \leq c \left (\left \Vert A z_{k +1}\right \Vert  +\left \Vert A z_{k -1}\right \Vert  +\left\Vert \delta{z}_{k} \right\Vert\right ) . \label{g2_k1}
	\end{equation}
	Let us estimate the vector $g_{3 ,k}$. If in representation of $g_{1 ,k }$ operator $A$ is replaced by identity operator, then we get $g_{3 ,k}$ (we assume that $\psi _{1}$ respectivly is replaced by $\psi _{3}$). According to this for $g_{1 ,k}$ analogously as for $g_{3 ,k }$ the following estimation is true
	\begin{equation}
		\left \Vert g_{3 ,k}\right \Vert  \leq c \left (\left \Vert z_{k +1}\right \Vert  +\left \Vert z_{k -1}\right \Vert \right ) \leq c \left (\left \Vert A z_{k +1}\right \Vert  +\left \Vert A z_{k -1}\right \Vert \right ) . \label{g3_k}
	\end{equation}
	If we consider that operator $C$ satisfies condition (\ref{E1.3}), and $N$ is bounded, then for $g_{4 ,k }$ we get
	\begin{equation}
		\left \vert g_{4 ,k}\right \vert  \leq c \left (\left \Vert A z_{k}\right \Vert  +\left\Vert \delta{z}_{k} \right\Vert \right )\,. \label{g4_k}
	\end{equation}
	If we take into account that operator $M ( \cdot ) $ satisfies Lipschitz condition, then for $g_{5 ,k }$ we obtain
	\begin{equation}
		\left \Vert g_{5 ,k}\right \Vert  \leq c \left \Vert z_{k}\right \Vert  +\left \Vert f_{k} -\overline{f}_{k}\right \Vert  \leq c \left \Vert A z_{k}\right \Vert  +\left \Vert f_{k} -\overline{f}_{k}\right \Vert  . \label{g5_k}
	\end{equation}
	Finally, according to inequalities (\ref{g1_k1}), (\ref{g2_k1}), (\ref{g3_k}), (\ref{g4_k}) and (\ref{g5_k}) for the vector $g_{k}$ the following estimation is valid
	\begin{align}\label{eq:gf_k}
		& \left \Vert g_{k}\right \Vert \leq \left \Vert f_{k} -\overline{f}_{k}\right \Vert +c \left (\left \Vert A z_{k +1}\right \Vert  +\left \Vert A z_{k}\right \Vert  +\left \Vert A z_{k -1}\right \Vert  +\genfrac{\Vert }{\Vert }{}{}{ \Delta z_{k}}{\tau } +\genfrac{\Vert }{\Vert }{}{}{ \Delta z_{k -1}}{\tau }\right )\text{.}
	\end{align}
	
	Let us introduce the following denotations
	\begin{align*}
		& \delta _{k}  = \left \Vert B^{1/2} z_{0}\right \Vert  +\genfrac{\Vert }{\Vert }{}{}{ \Delta z_{0}}{\tau } +\tau  \left \Vert B^{1/2} \frac{ \Delta z_{0}}{\tau }\right \Vert  +\tau  \sum \limits _{i =1}^{k}\left \Vert f_{i} -\overline{f}_{i}\right \Vert\text{,} \\
		& \varepsilon _{k}  = \left \Vert B^{1/2} z_{k}\right \Vert  +\genfrac{\Vert }{\Vert }{}{}{ \Delta z_{k -1}}{\tau }\text{.}
	\end{align*}
	From \eqref{eq:gf_k} according to (\ref{E1.6}) we have
	\begin{align}
		& \left \Vert g_{k}\right \Vert \leq \left \Vert f_{k} -\overline{f}_{k}\right \Vert +c \left( {\varepsilon}_{k+1} + {\varepsilon}_{k} + \left\Vert B^{1/2} z_{k -1}\right\Vert \right) \; . \label{g_k}
	\end{align}
	Using this inequality, from \eqref{zk2} we obtain
	\begin{align}
		&\left \Vert B^{1/2} z_{k +1}\right \Vert  \leq c\delta_{k} +c \tau  \sum \limits _{i =1}^{k}\left (\varepsilon _{i+1}  + \varepsilon _{i} +\left \Vert B^{1/2} z_{i -1}\right \Vert\right )\,.\label{eq_zk}
	\end{align}
	Using simple transformation from (\ref{eq_zk}) we get
	\begin{align}
		&\left \Vert B^{1/2} z_{k +1}\right \Vert  \leq c\delta_{k} +c \tau  \sum \limits _{i =1}^{k+1}\varepsilon _{i}\,.\label{eq1_zk}
	\end{align}
	Analogously, according to (\ref{g_k}) and $\left\Vert {z}_{0} \right\Vert \leq \left\Vert {B}^{1/2}{z}_{0} \right\Vert$, from (\ref{zk2_1}) follows
	\begin{align}
		& \genfrac{\Vert }{\Vert }{}{}{ \Delta z_{k}}{\tau } \leq c{\delta}_{k} +c \tau  \sum \limits _{i =1}^{k +1}\varepsilon _{i}\,. \label{eq1_zk1}
	\end{align}
	If we add inequalities (\ref{eq1_zk}) and (\ref{eq1_zk1}), we get
	\begin{align}
		& \varepsilon_{k+1} \leq c \delta_{k} +c \tau  \sum \limits _{i =1}^{k +1}\varepsilon_{i}\,.\label{eq1_zk2}
	\end{align}
	If we request that $\tau$ satisfies the condition $\tau  \leq q/c$ ($0 <q <1$), then from (\ref{eq1_zk}) we get
	\begin{align*}
		& \varepsilon_{k+1} \leq c \delta_{k} +c \tau  \sum \limits _{i =1}^{k}\varepsilon_{i}\,.
	\end{align*}
	
	\noindent From here by the induction can be obtained (discrete analogue of Gr\"{o}nwall's lemma)
	\begin{equation}
		\varepsilon _{k +1} \leq c\left (1 +c \tau \right )^{k -1} \left (\delta _{k} +\tau  \varepsilon _{1}\right )\,. \label{eq1_G}
	\end{equation}
	From \eqref{eq1_G}, using inequality \begin{equation*}\left \Vert B^{1/2} z_{1}\right \Vert  \leq \left \Vert B^{1/2} z_{0}\right \Vert  +\tau  \left \Vert B^{1/2} \frac{ \Delta z_{0}}{\tau }\right \Vert \text{,}
	\end{equation*}
	follows estimation (\ref{T2_2}).
\end{proof}

\section{Estimate of the error of approximate solution}\label{sec:section6}
In this section, using the results of the previous sections, we prove the theorem, which considers the error estimate for the approximate solution. This theorem represents an almost trivial result of the theorem proved in the previous section. However, estimation of the approximation error for scheme (\ref{eq1}) because of nonlinear terms, requires additional calculations.

Before formulating the theorem regarding the convergence (error estimate of the approximate solution) of the scheme (\ref{eq1}), we would like to make a note about the well-posedness of the problem (\ref{E1.1}), (\ref{E1.2}). We mean from the beginning that the original continuous problem is well-posed and the solution is sufficiently smooth. Obviously, we require the smoothness of the solution to find a convergence order (rate). If we demand the minimal smoothness which is necessary for the well-posedness of the problem, then the convergence is guaranteed, but we are not able to establish an order. If we increase the smoothness order by one unit, then the convergence rate is equal to one (in this, as well as in the previous case, it is sufficient to take $u_{1} =\varphi _{0} +\tau  \varphi _{1}$). However, a convergence rate becomes two if we level up smoothness by two and define starting vector $u_{1}$ using the following formula
\begin{equation}u_{1} =\varphi _{0} +\tau  \varphi _{1} +\frac{\tau ^{2}}{2} \varphi _{2} , \label{eq1_1}
\end{equation}
where ${\varphi}_{2} = {u}^{{\prime}{\prime}}\left( 0 \right)$, ${u}^{{\prime}{\prime}}\left( 0 \right)$ is defined from the equation \eqref{E1.1} via $\varphi_{0}$ and $\varphi_{1}$ (we assume that $\varphi _{0}\,, \varphi _{1} \in D (B)$).

The further increase of smoothness of the solution does not make sense, as the approximation order of the scheme (\ref{eq1}) is not more than two (obviously, the convergence order generally does not exceed the approximation order).

Let us formulate above stated as a theorem (below everywhere
$c$ denotes positive constant).

\begin{theorem}
	\label{theorem6.1}
	Let the problem (\ref{E1.1}),
	(\ref{E1.2}) be well-posed. Besides, the following conditions are fulfilled:
	\begin{enumerate}[label=(\alph*)]
		\item\label{theorem6.1_a} $\varphi _{0}\; ,\varphi _{1}\; ,\varphi _{2} \in D (B)$;
		\item\label{theorem6.1_b} solution $u (t)$ of problem (\ref{E1.1}), (\ref{E1.2}) is continuously differentiable
		to third order including and $u^{ \prime  \prime  \prime } \left (t\right )$ satisfies Lipschitz condition;
		\item\label{theorem6.1_c} $u^{ \prime  \prime } \left (t\right ) \in D (B)$ for every $t$ from $\left [0 ,\overline{t}\right ]$ and function $B u^{ \prime  \prime } \left (t\right )$ satisfy Lipschitz condition.
	\end{enumerate}
	Then for scheme (\ref{eq1}), (\ref{eq1_1})
	the following estimates are true
	\begin{equation}
		\max_{1 \leq k \leq n -1}\left (\left \Vert B^{1/2} \widetilde{z}_{k +1}\right \Vert  +\genfrac{\Vert }{\Vert }{}{}{ \Delta \widetilde{z}_{k}}{\tau }\right ) \leq c \tau ^{2} , \label{Th}
	\end{equation}
	where $\widetilde{z}_{k} =u (t_{k}) -u_{k}$ is an error of approximate solution, $ \Delta \widetilde{z}_{k} =\widetilde{z}_{k +1} -\widetilde{z}_{k}$ .
\end{theorem}
\begin{proof}
	Let us introduce the following notations:
	\begin{equation*}
		{\delta}{u}\left( {t}_{k} \right)=\frac{u\left( {t}_{k+1} \right)-u\left( {t}_{k-1} \right)}{2\tau}\,,\quad \hat{u}\left( {t}_{k} \right)=\frac{u\left( {t}_{k+1} \right)+u\left( {t}_{k-1} \right)}{2}\,.
	\end{equation*}
	Let us write down the equation (\ref{E1.1}) at point $t =t_{k}$ in the following form
	\begin{align}
		& \frac{ \Delta ^{2}u \left (t_{k -1}\right )}{\tau ^{2}} +a_{1} B {\delta}{u}\left( {t}_{k} \right) +a_{2} B \hat{u}\left( {t}_{k} \right) +\widetilde{\psi }_{1} \left (\zeta _{k -1} ,\zeta _{k +1}\right )A\hat{u}\left( {t}_{k} \right)  \nonumber  \\
		&  +{\delta}\psi_{2}\left( \zeta_{k} \right) A\hat{u}\left( {t}_{k} \right) +\widetilde{\psi }_{3} \left (\left \Vert u \left (t_{k -1}\right )\right \Vert ^{2} ,\left \Vert u \left (t_{k +1}\right )\right \Vert ^{2}\right )\hat{u}\left( {t}_{k} \right) \nonumber  \\
		&  +C u \left (t_{k}\right ) +N {\delta}{u}\left( {t}_{k} \right) +M \left (u \left (t_{k}\right )\right ) =f \left (t_{k}\right ) +r_{\tau } \left (t_{k}\right ) , \label{eq14}
	\end{align}
	where $\zeta _{k} =\left \Vert A^{1/2} u \left (t_{k}\right )\right \Vert ^{2}$ ($\zeta  (t) =\left \Vert A^{1/2} u \left (t\right )\right \Vert ^{2}$),
	\begin{equation}
		r_{\tau } \left (t_{k}\right ) =\sum _{j =0}^{5}r_{j ,\tau } \left (t_{k}\right )\; , \label{rtau_B1}
	\end{equation}
	and where
	\begin{align*}
		r_{0 ,\tau } \left (t_{k}\right ) &= \frac{ \Delta ^{2}u \left (t_{k -1}\right )}{\tau ^{2}} -u^{ \prime  \prime } \left (t_{k}\right )\,,\quad r_{6 ,\tau } \left (t_{k}\right )  = N \left ({\delta}{u}\left( {t}_{k} \right) -u^{ \prime } \left (t_{k}\right )\right )\,, \\
		r_{2 ,\tau } \left (t_{k}\right ) &= \frac{1}{2} \widetilde{\psi }_{1} \left (\zeta _{k -1} ,\zeta _{k +1}\right )A \left ( \Delta ^{2}u \left (t_{k -1}\right )\right ) +\left (\widetilde{\psi }_{1} \left (\zeta _{k -1} ,\zeta _{k +1}\right ) -\psi _{1} \left (\zeta _{k}\right )\right ) A u \left (t_{k}\right )\,, \\
		r_{3 ,\tau } \left (t_{k}\right )  &= a_{1} B \left ({\delta}{u}\left( {t}_{k} \right) -u^{ \prime } \left (t_{k}\right )\right )\,,\quad r_{1 ,\tau } \left (t_{k}\right ) = \frac{1}{2} a_{2} B \left ( \Delta ^{2}u \left (t_{k -1}\right )\right )\,, \\
		r_{4 ,\tau } \left (t_{k}\right ) &= {\delta}\psi_{2}\left( \zeta_{k} \right) \frac{1}{2} A \left ( \Delta ^{2}u \left (t_{k -1}\right )\right ) +\left ({\delta}\psi_{2}\left( \zeta_{k} \right) -{\left( \psi _{2} (\zeta  \left( t_{k} \right) \right)}^{\prime}_{t}\right ) A u \left (t_{k}\right )\,, \\
		r_{5 ,\tau } \left (t_{k}\right ) &= \widetilde{\psi }_{3} \left (\left \Vert u \left (t_{k -1}\right )\right \Vert ^{2} ,\left \Vert u \left (t_{k +1}\right )\right \Vert ^{2}\right ) \Delta ^{2}u \left (t_{k -1}\right ) \\
		&+\left (\widetilde{\psi }_{3} \left (\left \Vert u \left (t_{k -1}\right )\right \Vert ^{2} ,\left \Vert u \left (t_{k +1}\right )\right \Vert ^{2}\right )\right. \left.-\psi _{3} \left (\left \Vert u \left (t_{k}\right )\right \Vert ^{2}\right )\right )u \left (t_{k}\right )\,.
	\end{align*}
	From (\ref{eq14}) and (\ref{eq1})
	according to Theorem \ref{theorem5.1} we obtain
	\begin{align}
		& \left \Vert B^{1/2} \widetilde{z}_{k +1}\right \Vert  +\genfrac{\Vert }{\Vert }{}{}{ \Delta \widetilde{z}_{k}}{\tau } \nonumber \\
		&\leq c \left (\left \Vert B^{1/2} \widetilde{z}_{0}\right \Vert  +\genfrac{\Vert }{\Vert }{}{}{ \Delta \widetilde{z}_{0}}{\tau } +\tau  \left \Vert B^{1/2} \frac{ \Delta \widetilde{z}_{0}}{\tau }\right \Vert  +\tau  \sum \limits _{i =1}^{k}\left \Vert r_{\tau } \left (t_{k}\right )\right \Vert \right ) . \label{Azk}
	\end{align}
	If we carry out the routine calculations we obtain
	\begin{equation}
		\left \Vert r_{\tau } \left (t_{k}\right )\right \Vert  \leq c \tau ^{2} . \label{eq_R}
	\end{equation}
	It is obvious that according to the conditions \hyperref[theorem6.1_a]{(a)}, \hyperref[theorem6.1_b]{(b)} and \hyperref[theorem6.1_c]{(c)} of the \hyperref[theorem6.1]{Theorem \ref*{theorem6.1}}, and the equality \eqref{eq1_1} the following inequalities are true
	\begin{align}
		& \left \Vert B^{1/2} \left ( \Delta \widetilde{z}_{0}\right )\right \Vert + \genfrac{\Vert }{\Vert }{}{}{ \Delta \widetilde{z}_{0}}{\tau }  \leq c \tau ^{2}\,, \label{eq26}
	\end{align}
	From (\ref{Azk}), taking into account (\ref{eq_R}) and (\ref{eq26}), follows (\ref{Th}).
\end{proof}

\section{Iterative method for discrete problem}\label{sec:section7}
Let us rewrite equation (\ref{eq1}) in the following form
\begin{equation}
	T_{k} v_{k +1} =\frac{1}{2} \tau  \psi _{2} (\gamma _{k -1}) A v_{k +1} -\tau  N v_{k +1} +\widetilde{f}_{k}\; , \label{i1.1}
\end{equation}
where $v_{k +1} =(u_{k +1} +u_{k -1})/2$,
\begin{align*}
	& T_{k}  =  \left (2 +\tau ^{2} a_{3 ,k}\right ) I +\tau  \left (a_{1} +\tau  a_{2}\right ) B + \tau  {b}_{k} A\,,\quad {b}_{k}=\tau  a_{1 ,k} +\frac{1}{2} \psi _{2} \left( \gamma _{k +1} \right)\,, \\
	& \widetilde{f}_{k} = \tau  \widetilde{g}_{k} +\tau  a_{1} B u_{k -1} +2 u_{k}\,,\quad\widetilde{g}_{k}  = \tau  f_{k} -\tau  M \left (u_{k}\right ) +N u_{k -1} -\tau  C u_{k}\;\text{.}
\end{align*}
Equation (\ref{i1.1}) is solved by using the following iteration
\begin{equation}
	T_{k ,m} v_{k +1}^{(m +1)} =\frac{1}{2} \tau  \psi _{2} (\gamma _{k -1}) A v_{k +1}^{(m)} -\tau  N v_{k +1}^{(m)} +\widetilde{f}_{k}\; , \label{i1.3}
\end{equation}
where $m =0, 1, \ldots$,
\begin{align*}
	& T_{k ,m}  = \left (2 +\tau ^{2} a_{3 ,k}^{(m)}\right ) I +\tau  \left (a_{1} +\tau  a_{2}\right ) B +\tau {b}_{k}^{\left( m \right)} A\,,\quad {b}_{k}^{\left( m \right)}=\tau  a_{1 ,k}^{\left( m \right)} +\frac{1}{2} \psi _{2} \left( \gamma _{k +1}^{\left( m \right)} \right)\,,\\
	& a_{1 ,k}^{(m)} = \widetilde{\psi }_{1} \left (\gamma _{k -1} ,\gamma _{k +1}^{(m)}\right )\,,\quad\gamma _{k +1}^{(m)}  = \left \Vert A^{1/2} u_{k +1}^{(m)}\right \Vert ^{2} ,\text{\quad }u_{k +1}^{(m)} =2 v_{k +1}^{(m)} -u_{k -1}\;\text{,} \\
	& a_{3 ,k}^{(m)} = \widetilde{\psi }_{3} \left (\left \Vert u_{k -1}\right \Vert ^{2} ,\left \Vert u_{k +1}^{(m)}\right \Vert ^{2}\right ) ,\text{\quad }v_{k +1}^{(0)} =(u_{k} +u_{k -1})/2\;\text{.}
\end{align*}
\begin{remark}\label{remark7.1}
	We already required that $A$ and $B$ to be self-adjoint positively defined operators. Besides condition (\ref{E1.2.1}) is valid. From here follows that $D (A) \subset D (B^{1/2})$ (see  Remark \ref{remark1.1}). In order to show convergence of iterative method (\ref{i1.3}) we require the following condition to be fulfilled $D (A) =D (B^{1/2})$. From here follows that operator $B^{1/2} A^{ -1}$ (as it is closed operator defined in whole Hilbert space $H$) is bounded according to the closed graph theorem. Let us denote norm of this operator by $c_{1} =\left \Vert B^{1/2} A^{ -1}\right \Vert $.
\end{remark}
\begin{remark}\label{remark7.2}
	Condition $D (A) =D (B^{1/2})$ is automatically fulfilled for equation (\ref{B1}), if unknown function and its derivatives satisfy homogeneous boundary conditions. In this case $B^{1/2} =A$, where $A$ is an expansion of symmetrical operator ($-\partial _{x x}^{2}$) (with homogeneous boundary conditions) till a self-adjoint operator.
\end{remark}
\begin{remark}\label{remark7.3}
	Regarding iteration (\ref{i1.3}) it is important that $T_{k ,m}$ is self-adjoint positively defined operator ($a_{3 ,k}^{(m)}$, $a_{1 ,k}^{(m)}$and $\psi _{2} (\gamma _{k +1}^{(m)}) $ are nonnegative, as according to the condition $\psi _{1} (s)$, $\;\psi _{2} (s)$ and $\psi _{3} (s)$, $s \in [0 , +\infty [$, are nonnegative functions). It is well-known that from here follows that $R (T_{k ,m}) =H$, i.e. equation $T_{k ,m} u =g$ , for any vector $g$ from $H$ has unique solution  $u \in D (B)  $ and it continuously depends on  $g$.
\end{remark}
Let us prove convergence of iteration (\ref{i1.3}). The prove depends on many standard transformations. The first step is to prove uniform boundedness of vectors ${w}_{k+1}^{\left( m \right)}={B}^{1/2}{v}_{k+1}^{\left( m \right)}$ obtained by iteration (\ref{i1.3}). This fact is important, as it gives a certain opportunity that iteration (\ref{i1.3}) might be converged.

\begin{enumerate}[align=left,label=\textbf{Step~{\arabic*}.}]
	\item\label{st:step_one}
	Prove uniform boundedness of vector sequence ${w}_{k+1}^{(m)}$ obtained by using iteration (\ref{i1.3}).
	\begin{proof}
		Let us introduce the following notations:
		\begin{align*}
			&S_{k ,m} = \left (2 +\tau ^{2} a_{3 ,k}^{(m)}\right ) I +{a}_{\tau} B\,,\quad S_{k} = \left (2 +\tau ^{2} a_{3 ,k}\right ) I +{a}_{\tau} B\,,\quad {a}_{\tau}=\tau  \left (a_{1} +\tau  a_{2}\right )\,,\\
			&{P}_{k,m}={B}^{1/2}{T}^{-1}_{k,m}\,,\quad{Q}_{k,m}={B}^{1/2}{S}^{-1/2}_{k,m}\,,\quad{P}_{k}={B}^{1/2}{T}^{-1}_{k}\,,\quad
			{Q}_{k}={B}^{1/2}{S}^{-1/2}_{k}\,.
		\end{align*}
		If we define the vector $v^{\left( m+1 \right)}_{k+1}$ from the equation \eqref{i1.3}, after applying the operator $B^{1/2}$ and move on to the norm, we get
		\begin{align}
			\left \Vert {w}_{k +1}^{(m +1)}\right \Vert  &\leq \frac{1}{2} \tau  \psi _{2} (\gamma _{k -1}) \left \Vert {P}_{k,m}\right \Vert  \left \Vert A B^{ -1/2}\right \Vert  \left \Vert {w}_{k +1}^{(m)}\right \Vert  \nonumber  \\
			& +\tau  \left \Vert {P}_{k,m}\right \Vert  \left \Vert N\right \Vert  \left \Vert B^{ -1/2}\right \Vert  \left \Vert {w}_{k +1}^{(m)}\right \Vert +\left \Vert {P}_{k,m} \widetilde{f}_{k}\right \Vert  . \label{i2.9}
		\end{align}
		Operator $T_{k ,m}$ should be written in the following form
		\begin{equation}
			T_{k ,m} =S_{k ,m}^{1/2} G_{k ,m} S_{k ,m}^{1/2}\,,\quad G_{k ,m} = I +\tau  b_{k}^{(m)} S_{k ,m}^{ -1/2} A S_{k ,m}^{ -1/2}\,. \label{i3.5}
		\end{equation}
		According to (\ref{i3.5}) we have
		\begin{equation}
			\left \Vert {P}_{k,m}\right \Vert  \leq \left \Vert {Q}_{k,m}\right \Vert  \left \Vert G_{k ,m}^{ -1}\right \Vert  \left \Vert S_{k ,m}^{ -1/2}\right \Vert  . \label{i2.10}
		\end{equation}
		The following estimations can be obtained easily:
		\begin{equation}
			\left \Vert {Q}_{k,m}\right \Vert  \leq \frac{1}{\sqrt{\tau  a_{1}}}\,,\quad \left \Vert S_{k ,m}^{ -1/2}\right \Vert  \leq \frac{1}{\sqrt{2}}\; ,\text{\quad }\left \Vert G_{k ,m}^{ -1}\right \Vert  \leq 1\,. \label{i2.15}
		\end{equation}
		If we insert inequalities (\ref{i2.15}) into (\ref{i2.10}), we get
		\begin{equation}
			\left \Vert {P}_{k,m}\right \Vert = \left\Vert {B}^{1/2}{T}^{-1}_{k,m} \right\Vert  \leq \frac{1}{\sqrt{2 \tau  a_{1}}}\; . \label{i2.28}
		\end{equation}
		As $\gamma _{k}$ is uniformly bounded (see Remark \ref{remark2.2}), we have
		\begin{equation}\label{eq:psi_2gamma_k}
			\psi _{2} (\gamma _{k}) \leq M_{2} ,\;M_{2} =\max \psi _{2} (s) < +\infty  ,\;0 \leq s \leq \max_{k}\gamma _{k}\; < +\infty\,.
		\end{equation}
		From (\ref{i2.9}) according to estimations (\ref{i2.28}), (\ref{eq:psi_2gamma_k}), $\left\Vert {A}{B}^{-1/2} \right\Vert \leq {b}_{0}$ (see \eqref{E1.6}) and $\left\Vert {B}^{-1/2} \right\Vert \leq 1/\sqrt{m_{B}}$ ($m_{B}$ is a lower bound of the operator $B$) follows
		\begin{equation}
			\left \Vert {w}_{k +1}^{(m +1)}\right \Vert  \leq \sqrt{\tau } M_{3} \left \Vert {w}_{k +1}^{(m)}\right \Vert  +\left \Vert {P}_{k,m} \widetilde{f}_{k}\right \Vert  , \label{i2.29}
		\end{equation}
		where $M_{3}$ is a positive constant (independent of $m$ and $n$).
		
		Let us estimate norm of the vector
		\begin{align}
			& {P}_{k,m} \widetilde{f}_{k}  = \tau  {P}_{k,m} \widetilde{g}_{k} +\tau  a_{1} {P}_{k,m} B u_{k -1} +2 {P}_{k,m} u_{k}\; . \label{i2.29.1}
		\end{align}
		According to Theorem \ref{theorem2.1}, $A u_{k}$, $M \left (u_{k}\right )  $ and $u_{k}$ vectors are uniformly bounded,  (uniform boundeness of vectors $M \left (u_{k}\right )$ follows from (\ref{E2.4})), also according to  \eqref{E1.3} we have $\left \Vert C u_{k}\right \Vert  \leq a_{0} \left \Vert A u_{k}\right \Vert  $. From here follows that vectors $\widetilde{g}_{k}$ are uniformly bounded. According to this from (\ref{i2.28}) follows that there exists such constant $M_{0}$ (independent of $n$ and $m$), that
		\begin{equation}
			\sqrt{\tau } \left \Vert {P}_{k,m} \widetilde{g}_{k}\right \Vert  \leq \frac{1}{\sqrt{2 a_{1}}} \left \Vert \widetilde{g}_{k}\right \Vert  \leq M_{0}\; . \label{i2.29.5}
		\end{equation}
		Let us estimate second summand on the right-hand side of equality (\ref{i2.29.1}). Using formula (\ref{i3.5}) and inequalities (\ref{i2.15}) we get the following estimation
		\begin{align}
			&\tau  \left \Vert {P}_{k,m} B u_{k -1}\right \Vert = \tau  \left \Vert {Q}_{k,m} G_{k ,m}^{ -1} {Q}_{k,m} B^{1/2} u_{k -1}\right \Vert \leq \frac{1}{a_{1}} \left \Vert B^{1/2} u_{k -1}\right \Vert  . \label{i2.29.51}
		\end{align}
		From here according to Theorem \ref{theorem2.1} follows
		\begin{equation}
			\tau  \left \Vert {P}_{k,m} B u_{k -1}\right \Vert  \leq M_{1}\; , \label{i2.29.7}
		\end{equation}
		where $M_{1}$ is positive constant (independent of $n$ and $m$).
		
		Eventually, we need to estimate third summand in equality (\ref{i2.29.1}). Let us rewrite it in the following form
		\begin{equation}
			{P}_{k,m} u_{k} =B^{1/2} \left (T_{k ,m}^{ -1} -C_{k ,m}^{ -1}\right ) u_{k} +B^{1/2} C_{k ,m}^{ -1} u_{k}\; , \label{i3.5.1}
		\end{equation}
		where $C_{k ,m} =\left (2 +\tau ^{2} a_{3 ,k}^{(m)}\right ) I +\tau {b}_{k}^{\left( m \right)} A\;\text{.}$
		
		Obviously, we have
		\begin{align}
			& T_{k ,m}^{ -1} -C_{k ,m}^{ -1} = T_{k ,m}^{ -1} \left (C_{k ,m} -T_{k ,m}\right ) C_{k ,m}^{ -1} = -{a}_{\tau} T_{k ,m}^{ -1} B C_{k ,m}^{ -1}\; . \label{i3.5.3}
		\end{align}
		If we insert (\ref{i3.5.3}) into (\ref{i3.5.1}) and consider (\ref{i3.5}), we get
		\begin{align}
			{P}_{k,m} u_{k} &= \left (B^{1/2} A^{ -1}\right ) C_{k ,m}^{ -1} \left (A u_{k}\right ) \nonumber  \\
			& -{a}_{\tau} {Q}_{k,m} G_{k ,m}^{ -1} {Q}_{k,m} \left (B^{1/2} A^{ -1}\right ) C_{k ,m}^{ -1} \left (A u_{k}\right ) . \label{i3.7}
		\end{align}
		From (\ref{i3.7}) according to (\ref{i2.15}), $\left \Vert C_{k ,m}^{ -1}\right \Vert  \leq 1/2$, Remark \ref{remark7.1} and also \hyperref[theorem2.1]{Theorem \ref*{theorem2.1} } we have
		\begin{equation}
			\left \Vert {P}_{k,m} u_{k}\right \Vert  \leq M_{4} \left \Vert A u_{k}\right \Vert \leq {M}_{5}\,. \label{i2.31}
		\end{equation}
		where $M_{5}$ is  positive constant (independent of $n$ and $m$), $M_{4} =c_{1} \left (2 +\tau  a_{2}/a_{1}\right )$.
		
		Let us insert inequalities (\ref{i2.29.5}), (\ref{i2.29.7}) and (\ref{i2.31}) into (\ref{i2.29.1}), we get
		\begin{equation}
			\left \Vert {P}_{k,m} \widetilde{f}_{k}\right \Vert  \leq M_{6}\,,\quad M_{6} =\sqrt{\tau } M_{0} +a_{1} M_{1} +2 M_{5}\,. \label{i2.31.3}
		\end{equation}
		From (\ref{i2.29}) according to (\ref{i2.31.3}) we have
		\begin{equation}
			\left \Vert {w}_{k +1}^{(m +1)}\right \Vert  \leq \sqrt{\tau } M_{3} \left \Vert {w}_{k +1}^{(m)}\right \Vert  +M_{6}\; . \label{i2.35}
		\end{equation}
		Let $\tau $ satisfies condition $\sqrt{\tau } M_{3} \leq q <1$. Then from (\ref{i2.35}) follows
		\begin{equation}
			\left \Vert {w}_{k +1}^{(m)}\right \Vert  \leq q^{m} \left \Vert {w}_{k +1}^{(0)}\right \Vert  +\frac{M_{6}}{1 -q}\; . \label{i2.39}
		\end{equation}
		
		Inequality (\ref{i2.39}) shows, that vector sequence ${w}_{k +1}^{\left( m \right)}={B}^{1/2}{v}^{\left( m \right)}_{k+1}$ ($v_{k +1}^{(0)} =(u_{k +1}^{(0)} +u_{k -1})/2$, $u_{k +1}^{(0)} =u_{k}$) obtained by iterative method is uniformly bounded, i.e. there exists such constant $M_{7}$ (independent of $n$ and $m$) , that $\left \Vert {w}_{k +1}^{(m)}\right \Vert  \leq M_{7}$ (here we again used Theorem \ref{theorem2.1} for vector ${w}_{k +1}^{(0) }$).
	\end{proof}
    \item\label{st:step_two}
    Let us estimate norm of error $Z_{k +1}^{(m)} =B^{1/2} \left (v_{k +1} -v_{k +1}^{(m +1)}\right )$. Let us define $v_{k +1}$ and $v_{k +1}^{(m +1)  }$ from equalities (\ref{i1.1}) and (\ref{i1.3}), respectively. If we apply  $B^{1/2}$ to the difference of these vectors and perform corresponding transformation, we get
    \begin{align}
    	Z_{k +1}^{(m +1)} &= \frac{1}{2} \tau  \psi _{2} (\gamma _{k -1}) B^{1/2} L_{k ,m} {L} {w}_{k +1} +\frac{1}{2} \tau  \psi _{2} (\gamma _{k -1}) {P}_{k,m} {L} Z_{k +1}^{(m)} \nonumber\\
    	&-\tau  B^{1/2} L_{k ,m} {S} {w}_{k +1} -\tau  {P}_{k,m} {S} Z_{k +1}^{(m)} +B^{1/2} L_{k ,m} \widetilde{f}_{k}\; , \label{i2.43}
    \end{align}
    where ${L}={A}{B}^{-1/2}$, ${S}={N}{B}^{-1/2}$ and $L_{k ,m} =T_{k}^{ -1} -T_{k ,m}^{ -1}$.
    
    Let us note, that in representation of (\ref{i2.43}) the most complicated term is $B^{1/2} L_{k ,m} \widetilde{f}_{k}$, as it does not have a small parameter $\tau $ as a multiplier, that provides to obtain an estimation from where follows convergence of the iteration method (\ref{i1.3}). For this term, using simple transformation we get
    \begin{align}
    	B^{1/2} L_{k ,m} \widetilde{f}_{k} &= -{P}_{k} \left (T_{k} -T_{k ,m}\right ) T_{k ,m}^{ -1} \widetilde{f}_{k}= -\tau ^{2} {\zeta}_{3,k}^{\left( m \right)} {P}_{k} T_{k ,m}^{ -1} \widetilde{f}_{k} \nonumber  \\
    	& -\tau ^{2} {\zeta}_{1,k}^{\left( m \right)} {P}_{k} A T_{k ,m}^{ -1} \widetilde{f}_{k} -\frac{1}{2} \tau  {\zeta}_{2,k}^{\left( m \right)} {P}_{k} A T_{k ,m}^{ -1} \widetilde{f}_{k}\;, \label{i2.45}
    \end{align}
    where ${\zeta}_{2,k}^{\left( m \right)}=\psi _{2} (\gamma _{k +1}) -\psi _{2} (\gamma _{k +1}^{(m)})\,\text{,}\quad{\zeta}_{j,k}^{\left( m \right)}=a_{j ,k} -a_{j ,k}^{(m)}\,\text{,}\quad {j}=1,3$.
    
    Analogously to (\ref{i3.5}) we can represent operator $T_{k}$ in the following form:
    \begin{equation}
    	T_{k} =S_{k}^{1/2} G_{k} S_{k}^{1/2}\,,\quad {G}_{k} = I +\tau  b_{k} S_{k}^{ -1/2} A S_{k}^{ -1/2}\,. \label{i45.3}
    \end{equation}
    Analogously to inequalities (\ref{i2.15}) we have:
    \begin{equation}
    	\left \Vert {Q}_{k}\right \Vert  \leq \frac{1}{\sqrt{\tau  a_{1}}}\,,\quad \left \Vert S_{k}^{ -1/2}\right \Vert  \leq \frac{1}{\sqrt{2}}\; ,\text{\quad }\left \Vert G_{k}^{ -1}\right \Vert  \leq 1\,. \label{i45.7}
    \end{equation}
    From (\ref{i45.3}) using inequalities (\ref{i45.7}) follows
    \begin{equation}
    	\left \Vert {P}_{k}\right \Vert = \left\Vert {B}^{1/2}{T}^{-1}_{k} \right\Vert \leq \frac{1}{\sqrt{2 \tau  a_{1}}}\; . \label{i45.9.1}
    \end{equation}
    Using estimations (\ref{i45.7}), (\ref{i45.9.1}), (\ref{i2.15}), and (\ref{i2.28}), analogously to inequalities (\ref{i2.29.5}), (\ref{i2.29.51}) and (\ref{i2.31}), we can obtain:
    \begin{align}
    	& \tau  \left \Vert {P}_{k} A T_{k ,m}^{ -1} \right \Vert = \tau  \left \Vert {P}_{k} {L} {P}_{k,m} \right \Vert  \leq  \frac{b_{0}}{2 a_{1}}\,, \label{i45.9.3}
    \end{align}
    \begin{align}
    	\tau  \sqrt{\tau } \left \Vert {P}_{k} A T_{k ,m}^{ -1} B u_{k -1}\right \Vert &= \tau  {\nu}_{0} \left \Vert {L} {B}^{1/2} T_{k ,m}^{ -1} B u_{k -1}\right \Vert  \leq  b_{0} {\nu}_{0} a_{1}^{-1} \left \Vert B^{1/2} u_{k -1}\right \Vert  , \label{i45.9.5}
    \end{align}
    \begin{align}
    	& \sqrt{\tau } \left \Vert {P}_{k,m} A T_{k ,m}^{ -1} u_{k}\right \Vert  \leq \nu _{0} \left \Vert C_{k ,m}^{ -1} {A}{u}_{k} -{a}_{\tau} \left ({L}{B}^{1/2} T_{k ,m}^{ -1} B {A}^{-1} C_{k ,m}^{ -1}\right ) {A}{u}_{k}\right \Vert  \nonumber  \\
    	&  \leq \nu _{0} \left (1 +(a_{1} +\tau  a_{2}) b_{0} c_{1} a_{1}^{ -1}\right ) \left \Vert A u_{k}\right \Vert  ,\text{\quad }\nu _{0} =1/\sqrt{2 a_{1}} . \label{i45.9.7}
    \end{align}
    To obtain inequality (\ref{i45.9.7}) the following representations $T_{k ,m}^{ -1} = C_{k ,m}^{ -1}-{a}_{\tau} T_{k ,m}^{ -1} B C_{k ,m}^{ -1}$ and ${B}^{1/2}{T}_{k,m}^{-1}{B}^{1/2}={Q}_{k,m}{G}_{k,m}{Q}_{k,m}$ are used.
    
    From inequalities (\ref{i45.9.3}), (\ref{i45.9.5}) and (\ref{i45.9.7}), according to Theorem \ref{theorem2.1} it follows that there exists such positive constant $M_{8}$ (independent of $n$ and $m$) that
    \begin{equation}
    	\sqrt{\tau } \left \Vert {P}_{k} A T_{k ,m}^{ -1} \widetilde{f}_{k}\right \Vert  \leq M_{8}\; . \label{i45.9.11}
    \end{equation}
    According to the representation \eqref{i3.5} and the estimation \eqref{i2.28} we have:
    \begin{align}
    	&  \sqrt{\tau } \left \Vert {P}_{k} T_{k ,m}^{ -1} \right \Vert  \leq \frac{1}{2} \nu _{0}\,, \label{i45.9.15}
    \end{align}
    \begin{align}
    	&  \tau  \left \Vert {P}_{k} T_{k ,m}^{ -1}{B}{u} \right \Vert \leq  \frac{1}{2 a_{1}}\left\Vert {{B}^{1/2}}{u} \right\Vert\,,\quad {u}\in\operatorname{D}\left( B \right)\,.\label{i45.9.17}
    \end{align}
    From the inequalities (\ref{i45.9.15}) and \eqref{i45.9.17} analogously to (\ref{i45.9.11}) we have
    \begin{equation}
    	\sqrt{\tau } \left \Vert {P}_{k} T_{k ,m}^{ -1} \widetilde{f}_{k}\right \Vert  \leq M_{9}\; , \label{i45.9.29}
    \end{equation}
    where $M_{9}$ is a positive constant (independent of $n$ and $m$).
    
    If we apply norms in (\ref{i2.45}) and take into consideration the estimations (\ref{i45.9.11}) and (\ref{i45.9.29}), we get
    \begin{align}
    	& \left \Vert B^{1/2} L_{k ,m} \widetilde{f}_{k}\right \Vert  \leq  \sqrt{\tau }\left (\tau  M_{9} \left \vert {\zeta}_{3,k}^{\left( m \right)}\right \vert  +\tau  M_{8} \left \vert {\zeta}_{1,k}^{\left( m \right)}\right \vert + M_{8} \left \vert {\zeta}_{2,k}^{\left( m \right)}\right \vert \right ) . \label{i7.1}
    \end{align}
    
    Now, let us estimate norm of the operator $B^{1/2} L_{k ,m }$. From (\ref{i2.45}) we have
    \begin{align*}
    	& \left \Vert B^{1/2} L_{k ,m}\right \Vert  \leq  \tau ^{2} \left \vert {\zeta}_{3,k}^{\left( m \right)}\right \vert  \left \Vert {P}_{k} T_{k ,m}^{ -1}\right \Vert + \tau \left \Vert {P}_{k} A T_{k ,m}^{ -1}\right \Vert \left( \tau \left \vert {\zeta}_{1,k}^{\left( m \right)}\right \vert + \left \vert {\zeta}_{2,k}^{\left( m \right)}\right \vert \right) \text{.}
    \end{align*}
    From here according to inequalities (\ref{i45.9.3}) and (\ref{i45.9.15}) we have
    \begin{align}
    	& \left \Vert B^{1/2} L_{k ,m}\right \Vert  \leq \frac{1}{2} \tau ^{3/2} \nu _{0} \left \vert {\zeta}_{3,k}^{\left( m \right)}\right \vert + b_{0} \nu _{0}^{2} \left( \tau \left \vert {\zeta}_{1,k}^{\left( m \right)}\right \vert + \left \vert {\zeta}_{2,k}^{\left( m \right)}\right \vert \right)\,. \label{i7.3}
    \end{align}
    Let us note that, from uniform boundedness of vectors ${w}_{k +1}^{(m) }$ follows uniform boundedness of vectors $A v_{k +1}^{(m)}$ (see inequality (\ref{E1.6})). From here according to the inequality \eqref{eq:sqrt_A} follows uniform boundedness of vectors $A^{1/2} v_{k +1}^{(m) }$.
    
    If we insert the terms ${\gamma}_{k+1}^{\left( m \right)}$ and $\gamma_{k-1}$ into the inequality \eqref{a1_k1} instead of the following ones ${\gamma}_{k+1}$ and $\bar{\gamma}_{k-1}$, respectively, we obtain
    \begin{align}
    	\left \vert {\zeta}_{1,k}^{\left( m \right)}\right \vert  &\leq  K_{1} \left \vert \gamma _{k +1} -\gamma _{k +1}^{(m)}\right \vert \leq  2 K_{1} K_{2} \left \Vert A^{1/2} u_{k +1} -A^{1/2} u_{k +1}^{(m)}\right \Vert  \nonumber\\
    	&= 4 K_{1} K_{2} \left \Vert A^{1/2} v_{k +1} -A^{1/2} v_{k +1}^{(m)}\right \Vert  , \label{i7.9}
    \end{align}
    where $K_{1} =\max \left \vert \psi _{1}^{ \prime } \left (s\right )\right \vert  < +\infty  ,\;0 \leq s \leq K_{2} =\max_{(k ,m)}(\gamma _{k} ,\gamma _{k}^{(m)}) < +\infty \,.$
    
    From \eqref{eq:sqrt_A} using inequality (\ref{E1.6}) it follows that
    \begin{equation*}
    	\left \Vert A^{1/2} u\right \Vert  \leq \frac{b_{0}}{\sqrt{m_{A}}} \left \Vert B^{1/2} u\right \Vert  ,\text{\quad } \forall u \in D (B^{1/2})\text{.}
    \end{equation*}
    According to this inequality, from (\ref{i7.9}) we have
    \begin{equation}
    	\left \vert {\zeta}_{1,k}^{\left( m \right)}\right \vert  \leq K_{3} \left \Vert {w}_{k +1} -{w}_{k +1}^{(m)}\right \Vert  . \label{i7.11}
    \end{equation}
    where $K_{3}$ is a positive constant (independent of $n$ and $m$).
    
    Analogously we get
    \begin{equation}
    	\left \vert {\zeta}_{j,k}^{\left( m \right)}\right \vert  \leq K_{4} \left \Vert {w}_{k +1} -{w}_{k +1}^{(m)}\right \Vert  ,\text{\quad } j=2,3\,,\quad K_{4} =\operatorname{c o n s t} >0. \label{i7.15}
    \end{equation}
    If we insert (\ref{i7.11}) and (\ref{i7.15}) inequalities into  (\ref{i7.1}) and (\ref{i7.3}) we get:
    \begin{align}
    	&  \left \Vert B^{1/2} L_{k ,m} \widetilde{f}_{k}\right \Vert \leq  \sqrt{\tau } K_{5} \left \Vert {w}_{k +1} -{w}_{k +1}^{(m)}\right \Vert  ,\text{\quad }K_{5} =\operatorname{c o n s t} >0 , \label{i7.21}
    \end{align}
    
    \begin{equation}
    	\left \Vert B^{1/2} L_{k ,m}\right \Vert  \leq K_{6} \left \Vert {w}_{k +1} -{w}_{k +1}^{(m)}\right \Vert  ,\text{\quad }K_{6} =\operatorname{c o n s t} >0. \label{i7.23}
    \end{equation}
    
    If we move on to norms in the equality (\ref{i2.43}) and consider the inequalities (\ref{i7.21}), (\ref{i7.23}), \eqref{i2.28}, $\left\Vert {L} \right\Vert \leq {b}_{0}$ (see \eqref{E1.6}) and Theorem \ref{theorem2.1} we get
    \begin{equation}
    	\left \Vert Z_{k +1}^{(m +1)}\right \Vert  \leq \sqrt{\tau } K_{0} \left \Vert Z_{k +1}^{(m)}\right \Vert  , \label{i7.25}
    \end{equation}
    where $K_{0}$ is a positive constant (independent of $n$ and $m$).
    
    Let $\tau $ satisfies condition $\sqrt{\tau } K_{0} \leq q <1$. Then from (\ref{i7.25}) follows that $\left \Vert Z_{k +1}^{(m)}\right \Vert  \leq q^{m} \left \Vert Z_{k +1}^{(0)}\right \Vert  $,  which means that (\ref{i1.3}) iteration is convergent.
\end{enumerate}

So, we proved the following theorem
\begin{theorem}
	\label{theorem7.4}
	If operators $A$ and $B$ satisfy conditions from section 1 and $D (A) =D (B^{1/2})$, then (\ref{i1.3}) iteration converges with geometric progression speed.
\end{theorem}


\def\printchapternonum{}
\bibliographystyle{plain}
\bibliography{bibdata}

\begin{thebibliography}{10}

\bibitem{Ar}
A.~Arosio and S.~Panizzi.
\newblock {On the well-posedness of the Kirchhoff string}.
\newblock {\em {Trans. Am. Math. Soc.}}, 348(1):305--330, 1996.

\bibitem{Bak1}
G.~A. Baker.
\newblock {Error estimates for finite element methods for second order
  hyperbolic equations}.
\newblock {\em {SIAM J. Numer. Anal.}}, 13(4):564--576, 1976.

\bibitem{Bak2}
G.~A. Baker and J.~H. Bramble.
\newblock {Semidiscrete and single step fully discrete approximations for
  second order hyperbolic equations}.
\newblock {\em {RAIRO, Anal. Num\'er.}}, 13:75--100, 1979.

\bibitem{Bak3}
G.~A. Baker, V.~A. Dougalis, and S.~M. Serbin.
\newblock {An approximation theorem for second-order evolution equations}.
\newblock {\em {Numer. Math.}}, 35(2):127--142, 1980.

\bibitem{Bal}
L.~A. Bales.
\newblock {Semidiscrete and single step fully discrete finite element
  approximations for second order hyperbolic equations with nonsmooth
  solutions}.
\newblock {\em {RAIRO, Mod\'elisation Math. Anal. Num\'er.}}, 27(1):55--63,
  1993.

\bibitem{BL}
J.~M. Ball.
\newblock {Stability theory for an extensible beam}.
\newblock {\em {J. Differ. Equations}}, 14:399--418, 1973.

\bibitem{Ber}
S.~Bernstein.
\newblock {Sur une classe d'\'equations fonctionnelles aux d\'eriv\'ees
  partielles}.
\newblock {\em {Izv. Akad. Nauk SSSR, Ser. Mat.}}, 4:17--26, 1940.

\bibitem{BM}
L.~C. Berselli and R.~Manfrin.
\newblock {Linear perturbations of the Kirchhoff equation}.
\newblock {\em {Comput. Appl. Math.}}, 19(2):157--178, 2000.

\bibitem{Bi}
P.~Biler.
\newblock {Remark on the decay for damped string and beam equations}.
\newblock {\em {Nonlinear Anal., Theory Methods Appl.}}, 10:839--842, 1986.

\bibitem{CC}
S.~M. Choo and S.~K. Chung.
\newblock {Finite difference approximate solutions for the strongly damped
  extensible beam equations}.
\newblock {\em {Appl. Math. Comput.}}, 112(1):11--32, 2000.

\bibitem{SS}
I.~Christie and J.~M. Sanz-Serna.
\newblock {A Galerkin method for a nonlinear integro-differential wave system}.
\newblock {\em {Comput. Methods Appl. Mech. Eng.}}, 44:229--237, 1984.

\bibitem{DanSp2}
P.~D'Ancona and S.~Spagnolo.
\newblock {On an abstract weakly hyperbolic equation modelling the nonlinear
  vibrating string}.
\newblock In {\em Developments in partial differential equations and
  applications to mathematical physics. Proceedings of an international
  meeting, Ferrara, Italy, October 14-18, 1991}, pages 27--32. New York, NY:
  Plenum Press, 1992.

\bibitem{DanSp}
P.~D'Ancona and S.~Spagnolo.
\newblock {A class of nonlinear hyperbolic problems with global solutions}.
\newblock {\em {Arch. Ration. Mech. Anal.}}, 124(3):201--219, 1993.

\bibitem{Br}
E.~H. de~Brito.
\newblock {Decay estimates for the generalized damped extensible string and
  beam equation}.
\newblock {\em {Nonlinear Anal., Theory Methods Appl.}}, 8:1489--1496, 1984.

\bibitem{GCh}
T.~Geveci and I.~Christie.
\newblock {The convergence of a Galerkin approximation scheme for an extensible
  beam}.
\newblock {\em {RAIRO, Mod\'elisation Math. Anal. Num\'er.}}, 23(4):597--613,
  1989.

\bibitem{JANGVELADZE201669}
T.~Jangveladze, Z.~Kiguradze, and B.~Neta.
\newblock {\em {Numerical solutions of three classes of nonlinear parabolic
  integro-differential equations}}.
\newblock Amsterdam: Elsevier/Academic Press, 2015.

\bibitem{TK}
T.~Kato.
\newblock {\em {Perturbation theory for linear operators}}, volume 132.
\newblock Springer, Cham, 2nd edition, 1984.

\bibitem{Kac}
J.~Ka\v{c}ur.
\newblock {Application of Rothe's method to perturbed linear hyperbolic
  equations and variational inequalities}.
\newblock {\em {Czech. Math. J.}}, 34:92--106, 1984.

\bibitem{Kr}
S.~G. Krein.
\newblock {Linear equations in Banach space (Linejnye uravneniya v banakhovom
  prostranstve)}.
\newblock {Moskau: Verlag ``Nauka'', Hauptredaktion f\"ur
  physikalisch-mathematische Literatur. 104 S. R. 0.34}, 1971.

\bibitem{Lad}
O.~A. Ladyzhenskaya.
\newblock {On the solution of nonstationary operator equations}.
\newblock {\em {Transl., Ser. 2, Am. Math. Soc.}}, 65:200--236, 1967.

\bibitem{LR}
I.-S. Liu and M.~A. Rincon.
\newblock {Effect of moving boundaries on the vibrating elastic string}.
\newblock {\em {Appl. Numer. Math.}}, 47(2):159--172, 2003.

\bibitem{Man}
R.~Manfrin.
\newblock {Global solvability to the Kirchhoff equation for a new class of
  initial data}.
\newblock {\em {Port. Math. (N.S.)}}, 59(1):91--109, 2002.

\bibitem{Mat}
M.~P. Matos.
\newblock {Mathematical analysis of the nonlinear model for the vibrations of a
  string}.
\newblock {\em {Nonlinear Anal., Theory Methods Appl.}}, 17(12):1125--1137,
  1991.

\bibitem{Med}
L.~A. Medeiros.
\newblock {On a new class of nonlinear wave equations}.
\newblock {\em {J. Math. Anal. Appl.}}, 69:252--262, 1979.

\bibitem{Nish}
K.~Nishihara.
\newblock {On a global solution of some quasilinear hyperbolic equation}.
\newblock {\em {Tokyo J. Math.}}, 7:437--459, 1984.

\bibitem{Pan}
S.~Panizzi.
\newblock {Low regularity global solutions for nonlinear evolution equations of
  Kirchhoff type}.
\newblock {\em {J. Math. Anal. Appl.}}, 332(2):1195--1215, 2007.

\bibitem{Per}
J.~Peradze.
\newblock {An approximate algorithm for a Kirchhoff wave equation}.
\newblock {\em {SIAM J. Numer. Anal.}}, 47(3):2243--2268, 2009.

\bibitem{Pul}
M.~Pultar.
\newblock {Solutions of abstract hyperbolic equations by the Rothe method}.
\newblock {\em {Apl. Mat.}}, 29:23--39, 1984.

\bibitem{Reed}
M.~Reed and B.~Simon.
\newblock {Methods of modern mathematical physics. I: Functional analysis. Rev.
  and enl. ed}.
\newblock {New York etc.: Academic Press, A Subsidiary of Harcourt Brace
  Jovanovich, Publishers, XV, 400 p.}, 1980.

\bibitem{RJ}
Dzh.~L. Rogava.
\newblock {Poludiskretnye skhemy dlya operatornykh differential'nykh
  uravneni\u{i}}.
\newblock {\em Izdatel'stvo ``Tekhnicheskogo Universitet'', Tbilisi, 288 p.},
  1995.

\bibitem{R2}
J.~Rogava.
\newblock {The study of the stability of semidiscrete schemes by means of
  Chebyshev orthogonal polynomials}.
\newblock {\em GSSR Mecn. Akad. Moambe}, 83(3):545--548, 1976.

\bibitem{RTs2}
J.~Rogava and M.~Tsiklauri.
\newblock {Integral semi-discrete scheme for a Kirchhoff type abstract equation
  with the general nonlinearity}.
\newblock {\em {Appl. Math. Inform. Mech.}}, 14(2):18--34, 2009.

\bibitem{RTs}
J.~Rogava and M.~Tsiklauri.
\newblock {Convergence of a semi-discrete scheme for an abstract nonlinear
  second order evolution equation}.
\newblock {\em {Appl. Numer. Math.}}, 75:22--36, 2014.

\bibitem{Sob}
P.~E. Sobolevskij and L.~M. Chebotarova.
\newblock {Approximative L\"osung durch das Geradenverfahren des Cauchyproblems
  f\"ur eine abstrakte hyperbolische Gleichung}.
\newblock {\em {Izv. Vyssh. Uchebn. Zaved., Mat.}}, 1977(5(180)):103--116,
  1977.

\bibitem{Sege}
G.~Szeg\"o.
\newblock {\em {Orthogonal polynomials. 4th ed}}, volume~23.
\newblock Providence, RI: American Mathematical Society (AMS), 1975.

\bibitem{vashakidze2020application}
Zurab Vashakidze.
\newblock {An application of the Legendre polynomials for the numerical
  solution of the nonlinear dynamical Kirchhoff string equation}.
\newblock {\em {Mem. Differ. Equ. Math. Phys.}}, 79:107--119, 2020.

\bibitem{S0}
S.~Woinowsky-Krieger.
\newblock {The effect of an axial force on the vibration of hinged bars}.
\newblock {\em {J. Appl. Mech.}}, 17:35--36, 1950.

\end{thebibliography}

\section*{Authors' addresses:}

\begin{description}
  \item[{\Rogava}\orcidA{}] \hfill \\ Faculty of Exact and Natural Sciences, Ivane Javakhishvili Tbilisi State University (TSU), Ilia Vekua Institute of Applied Mathematics (VIAM), 2 University St., Tbilisi 0186, Georgia. \\ E-mail: \href{mailto:jemal.rogava@tsu.ge}{\textbf{jemal.rogava@tsu.ge}}
  \item[{\Tsiklauri}\orcidB{}] \hfill \\ Missouri University of Science and Technology, Electromagnetic Compatibility Laboratory, 4000 Enterprise Drive, Rolla, MO 65409, USA. \\ E-mails: \href{mailto:tsiklaurim@mst.edu}{\textbf{tsiklaurim@mst.edu}}, \href{mailto:mtsiklauri@gmail.com}{\textbf{mtsiklauri@gmail.com}}
  \item[{\Vashakidze}\orcidC{}] \hfill \\ Institute of Mathematics, School of Science and Technology, The University of Georgia (UG), 77a, M. Kostava st., Tbilisi 0171, Georgia; Ilia Vekua Institute of Applied Mathematics (VIAM) of Ivane Javakhishvili Tbilisi State University (TSU), 2, University St., Tbilisi 0186, Georgia. \\ E-mails: \href{mailto:zurab.vashakidze@gmail.com}{\textbf{zurab.vashakidze@gmail.com}}, \href{mailto:z.vashakidze@ug.edu.ge}{\textbf{z.vashakidze@ug.edu.ge}}
\end{description}

\end{document}